\documentclass[a4paper]{article}
\usepackage[latin1]{inputenc}
\usepackage[english]{babel}
\usepackage{amsmath}
\usepackage{amsfonts}
\usepackage{amssymb}
\usepackage{graphicx}
\usepackage{bm,xcolor,multirow,geometry}
\usepackage{tikz}
\usetikzlibrary{shapes,arrows}
\usepackage[title]{appendix}
\usepackage{algorithm}
\usepackage{enumerate}
\usepackage[noend]{algpseudocode}
\usepackage[round]{natbib}
\usepackage[colorlinks=true,citecolor=blue]{hyperref}
\usepackage{siunitx}
\usepackage{caption,booktabs, makecell}
\usepackage{lscape}
\usepackage{threeparttable,subcaption}
\usepackage{pgfplots}
\pgfplotsset{compat = 1.16}
\usepackage{pgfgantt}
\usepackage{tikz}
\usepackage{tikz-timing}
\usepackage{tikz-qtree}
\usepackage{ifthen}
\usepackage{fp}
\usepackage{setspace}

\newcolumntype{L}{>{$}l<{$}} 
\newcolumntype{C}{>{$}c<{$}}
\newcolumntype{R}{>{$}r<{$}}

\newtheorem{example}{Example}
\newtheorem{lemma}{Lemma}
\usepackage{multicol}
\usepackage{etoolbox}

\makeatletter
\newcommand*{\rom}[1]{\expandafter\@slowromancap\romannumeral #1@}
\makeatother

\makeatletter

\allowdisplaybreaks 

\patchcmd{\NAT@citex}
{\@citea\NAT@hyper@{%
		\NAT@nmfmt{\NAT@nm}%
		\hyper@natlinkbreak{\NAT@aysep\NAT@spacechar}{\@citeb\@extra@b@citeb}%
		\NAT@date}}
{\@citea\NAT@nmfmt{\NAT@nm}%
	\NAT@aysep\NAT@spacechar\NAT@hyper@{\NAT@date}}{}{}

\patchcmd{\NAT@citex}
{\@citea\NAT@hyper@{%
		\NAT@nmfmt{\NAT@nm}%
		\hyper@natlinkbreak{\NAT@spacechar\NAT@@open\if*#1*\else#1\NAT@spacechar\fi}%
		{\@citeb\@extra@b@citeb}%
		\NAT@date}}
{\@citea\NAT@nmfmt{\NAT@nm}%
	\NAT@spacechar\NAT@@open\if*#1*\else#1\NAT@spacechar\fi\NAT@hyper@{\NAT@date}}
{}{}

\tikzstyle{decision} = [diamond, draw, fill=blue!20,
text width=6em, text badly centered, node distance=3cm, inner sep=0pt]
\tikzstyle{block} = [rectangle, draw, fill=blue!20,
text width=8.5em, text centered, rounded corners, minimum height=4em]
\tikzstyle{line} = [draw, -latex']
\tikzstyle{cloud} = [draw, ellipse,fill=red!20, node distance=3cm,
minimum height=2em]

\definecolor{green1}{rgb}{.3, .6, 0}
\makeatother

\begin{document}
\setcounter{page}{1}    
\pagenumbering{arabic}

\title{Scheduling a single parallel-batching machine with non-identical job sizes and incompatible job families}

 \author{{\normalsize Fan Yang$^\text{a}$, Morteza Davari$^\text{b}$, Wenchao Wei$^\text{c}$, Ben Hermans$^\text{a}$, Roel Leus$^\text{a}$}\footnote{Corresponding author. E-mail address: Roel.Leus@kuleuven.be, tel.: +32 16 32 69 67, ORCID: 0000-0002-9215-3914.}
	\\{\small\em $^\text{a}$ORSTAT, Faculty of Economics and Business, KU Leuven, Belgium}
	\\{\small\em $^\text{b}$SKEMA Business School, Universit\'{e} C\^{o}te d'Azur, Lille, France}
	\\{\small\em $^\text{c}$School of Economics and Management, Beijing Jiaotong University, China}}
\date{}
\maketitle

\baselineskip 0.8cm
\begin{abstract}\baselineskip 0.8cm
We study the scheduling of jobs on a single parallel-batching machine with non-identical job sizes and incompatible job families.
Jobs from the same family have the same processing time and can be loaded into a batch, as long as the batch size respects the machine capacity.
The objective is to minimize the total weighted completion time.  The problem combines two classic combinatorial problems, namely bin packing and single machine scheduling.
We develop three new mixed-integer linear-programming
formulations, namely
an assignment-based formulation,
a time-indexed formulation (TIF), and a set-partitioning formulation (SPF).
We also propose a column generation (CG) algorithm for the SPF, which is the basis for a branch-and-price (B\&P) algorithm and a CG-based heuristic.
We develop a preprocessing method to reduce the formulation size.
A  heuristic framework based on proximity search is also developed using the TIF\@.
The SPF and B\&P can solve instances with non-unit and unit job durations to optimality with up to 80 and 150 jobs within reasonable runtime limits, respectively. The proposed heuristics perform better than the methods from the literature.
\\
\textit{Keywords}: scheduling, batch processing, non-identical job sizes, incompatible job families, integer programming
\end{abstract}

\section{Introduction}
\label{Sec:Introduction}

Batch processing is a production mode that occurs widely in practice, such as in semiconductor manufacturing~\citep{uzsoy1994,ham2017}, hospital sterilization~\citep{ozturk2014}, and 3D printing~\citep{kucukkoc2019}. Two categories are generally distinguished~\citep{mathirajan2006}, namely serial batching and parallel batching. In serial batching, jobs can be batched on a machine in case they require the same setups; the machine handles jobs in a batch in a serial manner. In parallel batching, multiple jobs can be processed simultaneously on machines by grouping them into batches.
This paper considers parallel batching; for the serial case we refer the reader to \cite{shen2012} and \cite{pei2019}.

This study is motivated by a quality inspection of metering devices in an electricity company. The company provides different types of metering devices to measure electricity consumption in various settings (both for families as well as for factories).
We will refer to a meter type as a job family.
To ensure their accuracy, the meters need testing on a parallel-batching machine before installation, and each meter type has its own testing configuration, such that different types cannot be tested together (incompatibility).
The company does not produce the meters itself, but orders them from third-party suppliers.  Each order (below also referred to as a ``job'') has a specific size (number of meters) and cannot be split: each order needs to be tested in the same test run so that, if a meter has a problem in use, it can be easily traced back to its source and the remainder of the order can be recalled.
Multiple jobs of the same type can be tested together in one batch, as long as their total size respects the machine capacity.
We consider a single parallel-batching machine, which means that all jobs in a batch start and complete processing together. All jobs from one family have the same processing duration.
The objective is to minimize the total weighted completion time. 
Our main motivation for choosing this objective rather than, for example, minimizing the makespan is that inventory costs, which are driven by the jobs' weighted completion times \citep{pinedo2016}, form an important KPI for the electricity company. 
In the three-field notation of \cite{graham1979}, 
the problem studied in this paper 
can be denoted as $1|p\text{-}batch,v_i,incompatible|\sum w_iC_i$, where $p\text{-}batch$ refers to parallel batching, $v_i$ is the size of job $i$, $incompatible$ stands for incompatibilities among job families, and $w_i$ and $C_i$ denote the job weight and the completion time of job $i$, respectively.

\cite{dobson2001} study the same problem and refer to it as the ``batch loading and scheduling problem'' (BLSP). Batch loading refers to grouping jobs from the same family into batches, and batch scheduling is the sequencing of the batches to minimize the total weighted completion time.
\citeauthor{dobson2001} point out that the problem is NP-hard even if there is only a single family and if all job weights are proportional to the job sizes.
The authors also show that if for every two jobs~$i$ and~$j$ within the same family the job size~$v_i$ is an integer multiple of~$v_j$ whenever~$v_i \geq v_j$, and if the machine capacity and job sizes satisfy additional restrictions, then the problem is polynomially solvable.
\citeauthor{dobson2001} further propose a non-linear mixed-integer formulation, a lower bound, and three heuristics.
\cite{azizoglu2001} develop a branch-and-bound (B\&B) algorithm for this problem, which can solve up to 25 jobs in reasonable computation times.
\cite{koh2005} model a bottleneck operation in the production of a multi-layer ceramic capacitor as a BLSP, with objectives of makespan, total completion time, and total weighted completion time.
They design a number of heuristics, including a hybrid genetic algorithm.  In fact, makespan minimization in this setting simply equates with minimizing the number of batches for each family separately, such that minimizing the makespan can be achieved by solving a bin packing problem for each family.  The makespan objective is therefore not the most interesting objective from an optimization viewpoint, because the batching decision can be made via bin packing, while sequencing choices are irrelevant for this objective.
More generally, BLSP can be seen as a suitable combination of two classic combinatorial problems, namely bin packing and single machine scheduling, and in this way this study also adheres to a recent trend in literature in which this type of  combinations is studied (see, for example, \cite{wang2012} or \cite{nip2015}).

In follow-up work on BLSP, \cite{kashan2008}
propose a series of ant colony frameworks embedded with heuristic information, which perform better than the hybrid genetic algorithm of~\cite{koh2005}.
\cite{kempf1998} also include secondary resource constraints related to a burn-in operation from the testing stage in semiconductor industry. They develop  heuristics to minimize makespan and total completion time on a single parallel-batching machine.
As the discussion above demonstrates, multiple (meta)heuristics have been developed for the BLSP, while exact solution methods are only rarely considered. In this paper, we introduce a number of mixed-integer programming formulations to solve the BLSP exactly.

The BLSP has also been extended from a single machine to parallel machines. \cite{koh2004}, for example, design heuristics, including a genetic algorithm, for minimization of makespan, total completion time, and total weighted completion time. To minimize the makespan, \cite{jia2016} use ant colony optimization to form batches and a multi-fit algorithm~\citep{coffman1978} to schedule the batches, which outperforms  the methods in \cite{koh2004}. \cite{cakici2013} include job release dates and develop a variable neighborhood search scheme.
\cite{ham2017} consider a batch processing problem with job release dates, non-identical job sizes and incompatible families on parallel batching machines. The authors show that, with a fixed time limit, a constraint programming model can provide better solutions than two MILP models and a variable neighborhood search.
Many other scheduling environments and other objectives have also been studied for batch processing;
surveys can be found in \cite{potts2000}, \cite{mathirajan2006}, and \cite{monch2011}.

The main contributions of this paper are in three dimensions. First, we present three new linear formulations for the BLSP: an assignment-based formulation (ABF), a time-indexed formulation (TIF) and a set-partitioning formulation (SPF).\@ Among these formulations, the SPF performs best and can solve up to 150 jobs with unit durations and 80 jobs with non-unit durations within reasonable runtime limits.
Second, a column generation (CG) algorithm is proposed for the SPF, which is the basis of a branch-and-price (B\&P) method and a CG-based heuristic.
Third, a preprocessing method and a proximity search heuristic are developed.

The remainder of this paper is structured as follows.
Section~\ref{sec:Problem statement} contains the problem statement and three formulations.
In Section~\ref{sec:Column generation} we introduce the CG, the B\&P, and a CG-based heuristic.
Section~\ref{sec:Preprocessing} provides a preprocessing method to reduce the formulation size.
In Section~\ref{sec:Proximity search} we describe the proximity search.
The computational results are reported in Section~\ref{sec:Computational results}, and
 in Section~\ref{sec:Conclusions} we give a summary and some conclusions.

\section{Problem statement and linear formulations}
\label{sec:Problem statement}
We first provide some definitions and a problem statement in Section~\ref{subsec:Problem statement}. Subsequently, an assignment-based formulation, a time-indexed formulation and a set-partitioning formulation are introduced in Sections~\ref{subsec:Assignment-based formulation},~\ref{subsec:Time-indexed formulation} and~\ref{subsec:Set-partitioning formulation}, respectively.
\subsection{Definitions and problem statement}
\label{subsec:Problem statement}
We schedule a set $N=\{1,\ldots,n\}$ of $n$ independent jobs partitioned in $m$ job families. Each family $j$ is in the set $\mathcal{F}=\{1,\ldots, m\}$,
and coincides with a set of jobs $N_j\subseteq N$, where $\cup_{j=1}^m N_j=N$ and $N_i \cap N_j = \emptyset$ for all $\{i,j\}\subset \mathcal{F}$.
Each job $i$ has a weight $w_i$, a size $v_i$, and a processing time $p_i$. The processing time is the same for all jobs in each family~$j\in \mathcal{F}$: we have $p_i =q_j$ for all $i\in N_j$.
There is a single parallel-batching machine with capacity $V$, which can process several jobs from the same family simultaneously.
We assume $0\leq v_i\leq V$ for all $i\in N$.
The total size of the jobs in one batch cannot exceed the machine capacity.
The processing time of a batch containing jobs from family $j$ is $q_j$, and the completion time of all jobs in a batch equals the completion time of the batch.
Once the processing of a batch starts, there is no interruption and other jobs cannot be added to the batch before completion. A job cannot be split across batches. The objective is to minimize the total weighted completion time.

\begin{example}
\label{example:problem statement}
\normalfont
Consider an instance with two job families A and B ($m=2$), and six jobs ($n=6$). The machine capacity is $V =  4$. The job data are in Table~\ref{table:Example}. Figure~\ref{figure:example 1} depicts several feasible solutions in a Gantt-chart setup, in which time is the horizontal dimension. Each large box represents the execution of a (possibly empty) batch, while the gray boxes correspond to jobs (identified by their number). Solution~1, for example, starts with the processing of the first batch, which contains jobs 1 and 2. Both families have the same (unit) processing time, so that the completion time of the $k$-th batch simply equals $k$.
The total weight of the non-empty batches is 40, 30, 11, and 10, respectively, so that the objective value of
solution 1 equals $1\times 40 + 2 \times 30 + 3 \times 11 + 4 \times 10 = 173$. Similarly, we obtain the objective values of solutions 2 and 3 as 181 and 274.

\begin{table}[t]
	\centering
	\caption{ Job data for Example~\ref{example:problem statement}}
	\begin{tabular}{ccccc}
		\toprule[1pt]
		job $i$ & family & processing time $p_i$ & weight $w_i$ &size $v_i$ \\ \midrule
		1&A&1&20&1   \\
		2&A&1&20&1   \\
		3&A&1&11&3  \\
		4&A&1&10&3   \\
		5&B&1&10&2   \\
		6&B&1&20&2\\
		\bottomrule[1pt]
	\end{tabular}
	\label{table:Example}
\end{table}

\pgfmathsetlengthmacro\timeunit{3*6pt}
\pgfmathsetlengthmacro\vertunit{16.18034pt}
\begin{figure}[t]
	\centering
	\begin{tikzpicture}
	
	\node at (-2,2) [yshift=-1.6ex, xshift=-1ex, anchor=west]{Solution 1};
	
	\node at (-2,1.1) [yshift=-1.6ex, xshift=-1ex, anchor=west]{Solution 2};
	
	\node at (-2,0.1) [yshift=-1.6ex, xshift=-1ex, anchor=west]{Solution 3};

	\node at (0.6*\timeunit,1.4) [rectangle, anchor=south west, minimum height = \vertunit, draw, minimum width = 2.4*\timeunit, fill=white!10]{};
	\node at (0.6*\timeunit,1.4) [rectangle, anchor=south west, minimum height = \vertunit, draw, minimum width = 0.6*\timeunit, fill=gray!50]{1};
	\node at (1.2*\timeunit,1.4) [rectangle, anchor=south west, minimum height = \vertunit, draw, minimum width = 0.6*\timeunit, fill=gray!50]{2};

	\node at (3.6*\timeunit,1.4) [rectangle, anchor=south west, minimum height = \vertunit, draw, minimum width = 2.4*\timeunit, fill=white!10]{};
	\node at (3.6*\timeunit,1.4) [rectangle, anchor=south west, minimum height = \vertunit, draw, minimum width = 1.2*\timeunit, fill=gray!50]{5};
	\node at (4.8*\timeunit,1.4) [rectangle, anchor=south west, minimum height = \vertunit, draw, minimum width = 1.2*\timeunit, fill=gray!50]{6};
	
	\node at (6.6*\timeunit,1.4) [rectangle, anchor=south west, minimum height = \vertunit, draw, minimum width = 2.4*\timeunit, fill=white!10]{};
	\node at (6.6*\timeunit,1.4) [rectangle, anchor=south west, minimum height = \vertunit, draw, minimum width = 1.8*\timeunit, fill=gray!50]{3};
	\node at (9.6*\timeunit,1.4) [rectangle, anchor=south west, minimum height = \vertunit, draw, minimum width = 2.4*\timeunit, fill=white!10]{};
	\node at (9.6*\timeunit,1.4) [rectangle, anchor=south west, minimum height = \vertunit, draw, minimum width = 1.8*\timeunit, fill=gray!50]{4};
	\node at (12.6*\timeunit,1.4) [rectangle, anchor=south west, minimum height = \vertunit, draw, minimum width = 2.4*\timeunit, fill=white!10]{};
	\node at (15.6*\timeunit,1.4) [rectangle, anchor=south west, minimum height = \vertunit, draw, minimum width = 2.4*\timeunit, fill=white!10]{};

	\node at (0.6*\timeunit,0.5) [rectangle, anchor=south west, minimum height = \vertunit, draw, minimum width = 2.4*\timeunit, fill=white!10]{};
	\node at (0.6*\timeunit,0.5) [rectangle, anchor=south west, minimum height = \vertunit, draw, minimum width = 0.6*\timeunit, fill=gray!50]{1};
	\node at (1.2*\timeunit,0.5) [rectangle, anchor=south west, minimum height = \vertunit, draw, minimum width = 1.8*\timeunit, fill=gray!50]{3};

	\node at (3.6*\timeunit,0.5) [rectangle, anchor=south west, minimum height = \vertunit, draw, minimum width = 2.4*\timeunit, fill=white!10]{};
	\node at (3.6*\timeunit,0.5) [rectangle, anchor=south west, minimum height = \vertunit, draw, minimum width = 0.6*\timeunit, fill=gray!50]{2};
	\node at (4.2*\timeunit,0.5) [rectangle, anchor=south west, minimum height = \vertunit, draw, minimum width = 1.8*\timeunit, fill=gray!50]{4};

	\node at (6.6*\timeunit,0.5) [rectangle, anchor=south west, minimum height = \vertunit, draw, minimum width = 2.4*\timeunit, fill=white!10]{};
	\node at (6.6*\timeunit,0.5) [rectangle, anchor=south west, minimum height = \vertunit, draw, minimum width = 1.2*\timeunit, fill=gray!50]{5};
	\node at (7.8*\timeunit,0.5) [rectangle, anchor=south west, minimum height = \vertunit, draw, minimum width = 1.2*\timeunit, fill=gray!50]{6};
	
	\node at (9.6*\timeunit,0.5) [rectangle, anchor=south west, minimum height = \vertunit, draw, minimum width = 2.4*\timeunit, fill=white!10]{};
	\node at (12.6*\timeunit,0.5) [rectangle, anchor=south west, minimum height = \vertunit, draw, minimum width = 2.4*\timeunit, fill=white!10]{};
	\node at (15.6*\timeunit,0.5) [rectangle, anchor=south west, minimum height = \vertunit, draw, minimum width = 2.4*\timeunit, fill=white!10]{};

\node at (0.6*\timeunit,-0.4) [rectangle, anchor=south west, minimum height = \vertunit, draw, minimum width = 2.4*\timeunit, fill=white!10]{};
	\node at (0.6*\timeunit,-0.4) [rectangle, anchor=south west, minimum height = \vertunit, draw, minimum width = 0.6*\timeunit, fill=gray!50]{1};

\node at (3.6*\timeunit,-0.4) [rectangle, anchor=south west, minimum height = \vertunit, draw, minimum width = 2.4*\timeunit, fill=white!10]{};
\node at (3.6*\timeunit,-0.4) [rectangle, anchor=south west, minimum height = \vertunit, draw, minimum width = 0.6*\timeunit, fill=gray!50]{2};
\node at (6.6*\timeunit,-0.4) [rectangle, anchor=south west, minimum height = \vertunit, draw, minimum width = 2.4*\timeunit, fill=white!10]{};
\node at (6.6*\timeunit,-0.4) [rectangle, anchor=south west, minimum height = \vertunit, draw, minimum width = 1.2*\timeunit, fill=gray!50]{6};
\node at (9.6*\timeunit,-0.4) [rectangle, anchor=south west, minimum height = \vertunit, draw, minimum width = 2.4*\timeunit, fill=white!10]{};
\node at (9.6*\timeunit,-0.4) [rectangle, anchor=south west, minimum height = \vertunit, draw, minimum width = 1.8*\timeunit, fill=gray!50]{3};
\node at (12.6*\timeunit,-0.4) [rectangle, anchor=south west, minimum height = \vertunit, draw, minimum width = 2.4*\timeunit, fill=white!10]{};
\node at (12.6*\timeunit,-0.4) [rectangle, anchor=south west, minimum height = \vertunit, draw, minimum width = 1.8*\timeunit, fill=gray!50]{4};
\node at (15.6*\timeunit,-0.4) [rectangle, anchor=south west, minimum height = \vertunit, draw, minimum width = 2.4*\timeunit, fill=white!10]{};
\node at (15.6*\timeunit,-0.4) [rectangle, anchor=south west, minimum height = \vertunit, draw, minimum width = 1.2*\timeunit, fill=gray!50]{5};
	
	\end{tikzpicture}
	\caption{Three feasible solutions for Example~\ref{example:problem statement}}
	\label{figure:example 1}
\end{figure}
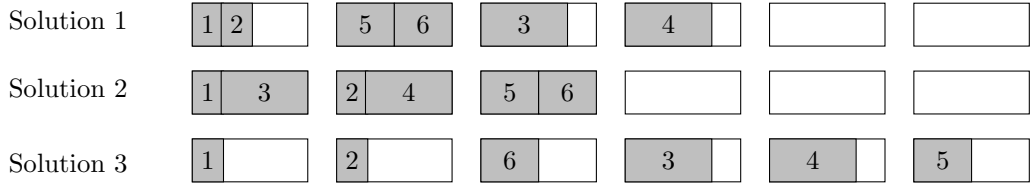

This example illustrates that an intuitive two-stage approach to BLSP consisting in first running a bin packing routine per family (where a bin equates with a batch) and then sequencing the batches, can miss the optimum: despite the lower number of batches, solution 2 is not as good as solution 1.  In fact, solution 1 is an optimal solution for this instance.  The three solutions are all three optimal for the given batching choice; the remaining sequencing problem is optimally solved following the weighted shortest processing time (WSPT) rule \citep{smith1956various}, which in this case simply puts the batches with largest total weight first.  This shows that an algorithm for this problem needs to coordinate the batching and sequencing decisions.

Moreover, solving the problem for each family separately and then combining the resulting batches into a single schedule by using the WSPT rule can also be sub-optimal. Combining jobs~1 and~3 in a first batch and jobs~2 and~4 in a second batch, for instance, constitutes an optimal solution when considering family~A only; and combining jobs~5 and~6 is also trivially optimal when considering family~B only. Applying the WSPT rule to these batches, however, leads to the sub-optimal solution~2.
Conversely, adopting the batching choices for one family for the joint problem with two families can also be sub-optimal for the instance where only one family is considered.
For example, the job combinations of family A in solution 1 cannot form an optimal solution when only considering family A.
\end{example}

In previous literature, \cite{dobson2001} and~\cite{koh2005} both describe a non-linear formulation for BLSP\@.  \cite{ham2017} study the generalization of BLSP to parallel machines,and present a linear assignment-based formulation as well as a time-indexed formulation, where big-M-type constraints are used to obtain a linear formulation.
Below, we give a linear assignment-based and time-indexed formulation that avoid such big-M constraints. Additionally,
we describe a set-partitioning formulation that exploits the link between BLSP and the bin packing problem~\citep{delorme2016}.

\subsection{Assignment-based formulation}
\label{subsec:Assignment-based formulation}
We first propose an assignment-based formulation.
Let batch~$k$ be in the $k$-th position of the sequence.  We work with a set of batches $\Omega=\{1,\ldots,n\}$.
Let the $(n+1)$-th batch be a dummy batch and define the set $\Omega^*= \Omega \cup \{n+1\}$ .
Clearly, there are at most $n$ non-empty batches in a solution.
We introduce binary variables~$x_{ik}$, which are one if job~$i$ is assigned to one of the batches~$k,k+1,\ldots,n$, and zero if job~$i$ is assigned to one of the batches~$1,2, \ldots,k-1$, with $i\in N$ and $ k\in \Omega^*$. Thus, job $i$ is in batch $k$ if $x_{ik}-x_{i(k+1)}=1$. We also introduce binary variables~$y_{jk}$, which are one if jobs of family $j$ are assigned to batch~$k$, and zero otherwise, for $j\in \mathcal{F}$, $k\in \Omega$. Let $\phi_i$ denote the family of job $i$, i.e., $\phi_i=j$ if $i \in N_j$. We define the continuous variables $u_k \in[0,1]$, which will be one if one or more jobs are assigned to batch~$k$, and  zero otherwise, for all $k\in \Omega$. Finally, $C_i$ is the completion time of job $i\in N$.
\begin{align}
\label{obj:Step}
\min\ &\displaystyle \sum_{i \in N} w_iC_i\\
s.t.\ & \displaystyle \sum_{j \in \cal{F}} y_{jk} = 1,  && \forall k \in \Omega, \label{cons:Step_01} \\
& \displaystyle \sum_{k \in \Omega} y_{jk} = |N_j|,  && \forall j \in \cal{F}, \label{cons:Step_02} \\
& x_{i1} = 1, &&\forall i\in N,  \label{cons:Step_03}\\
& x_{i(n+1)} = 0, &&\forall i\in N,  \label{cons:Step_04}\\
& x_{ik} \ge x_{i(k+1)}, && \forall i\in N,  k\in\Omega, \label{cons:Step_06}\\
& y_{\phi_i,k} \ge x_{ik} - x_{i(k+1)},  && \forall i \in N, k \in \Omega, \label{cons:Step_07}\\
&\displaystyle \sum_{i\in N} v_i \left(x_{ik} - x_{i(k+1)}\right)\leq V,   &&\forall k\in\Omega, \label{cons:Step_08}\\
& \displaystyle C_i \ge \sum_{k \in \Omega} \sum_{j \in \mathcal{F}} q_{j} x_{ik} y_{jk},   && \forall i \in N, \label{cons:Step_09}\\
&u_k \geq x_{ik} - x_{i(k+1)},  && \forall i\in N,  k\in\Omega, \label{cons:Step_10}\\
&\displaystyle u_k \leq \sum_{i\in N} \left(x_{ik} - x_{i(k+1)}\right),&& \forall k \in \Omega, \label{cons:Step_11}\\
&u_k \geq u_{k+1}, && \forall k \in \Omega\setminus \{n\}. \label{cons:Step_12}
\end{align}
In the above formulation, the set of constraints~\eqref{cons:Step_01} imposes that each batch be associated with exactly one family. Constraints~\eqref{cons:Step_02} assign $|N_j|$ batches to each family.
Constraints \eqref{cons:Step_03}--\eqref{cons:Step_06} ensure that each job is assigned to exactly one batch. Constraints~\eqref{cons:Step_07} associate batch $k$ to family $\phi_i$ if job $i$ has been assigned to batch~$k$.  Constraints~\eqref{cons:Step_08} impose the capacity constraint on each batch, and constraints \eqref{cons:Step_09} compute the job completion times.
Constraints~\eqref{cons:Step_10} ensure that $u_k = 1$ if at least one job~$i$ is assigned to batch~$k$. Constraints~\eqref{cons:Step_11} make sure  that if no job is assigned to batch $k$, then $u_k=0$.
Constraints~\eqref{cons:Step_12} guarantee that no empty batch exists between two used batches.  The constraints \eqref{cons:Step_10}--\eqref{cons:Step_12} are not strictly necessary, but help to restrict the solution space.

Note that in the above formulation, the term $ x_{ik} y_{jk} $ in constraints~\eqref{cons:Step_09} is not linear. The formulation can be linearized by introducing the auxiliary variables $\tau_{ikj}\in [0, 1]$. With these, constraints~\eqref{cons:Step_09} are replaced by the following constraints:
\begin{align}
& \displaystyle C_i \ge \sum_{k \in \Omega}\sum_{j \in \cal{F}} q_{j} \tau_{ikj},  && \forall i \in N, \label{cons:Step_13} \\
& \tau_{ikj} \le x_{ik},  && \forall i\in N,  k\in\Omega,j\in\cal{F}, \label{cons:Step_14} \\
& \tau_{ikj} \le y_{jk},   && \forall i\in N,  k\in\Omega,j\in\cal{F}, \label{cons:Step_15} \\
& \tau_{ikj} \ge x_{ik} +  y_{jk} -1,  && \forall i\in N,  k\in\Omega,j\in\cal{F}.  \label{cons:Step_16}
\end{align}


\subsection{Time-indexed formulation}
\label{subsec:Time-indexed formulation}

Based on the work by \cite{cakici2013} and \cite{ham2017}, we obtain an adapted TIF for BLSP which contains `big-M' type constraints; this formulation, henceforth referred to as TIFM, is given in Appendix~\ref{Appendix:TIF2017}.
In this subsection we present a new TIF, which will turn out to perform better in our computational experiments.
We define the time horizon $H=\{1,\ldots,H_{max}\}$, where $H$ contains $H_{max}=\sum_{i\in N}p_i$ time periods and time period $t$ runs from time $t-1$ to $t$.
We use binary variables $x_{it}$, which are one if the processing of job $i$ starts at time period $t$, and zero otherwise, for $i\in N$ and $t \in H_{\phi_i}$, where $H_j=\{1,\ldots,H_{max}-q_j+1\}$ is the set of relevant
starting periods for family $j$.  The binary variables $y_{jt}$ are 1 if jobs in family $j$ start processing at time period $t$, and zero otherwise, for $t \in H_j$ and $j\in \cal{F}$.
\begin{align}
\label{obj:newTIF}
\min \ &\displaystyle \sum_{i \in N} w_i \sum_{t\in H_{\phi_i}} x_{it}(t+p_i-1)\\
s.t. \ & \sum_{t \in H_{\phi_i}} x_{it} =1,
&&\forall i \in N,\label{cons:NewTIF_01}\\
&\sum_{j \in \mathcal{F}} \sum_{\tau=\max\{t-q_j+1,1\}}^{\min\{t, |H_j|\}}  y_{j\tau} \leq 1,
&&\forall t \in H,\label{cons:NewTIF_02}\\
&\sum_{i\in N_j} x_{it}v_i \leq Vy_{jt},
&&\forall t \in H_j,j\in \cal{F}.\label{cons:NewTIF_03}
\end{align}
The objective~(\ref{obj:newTIF}) is to minimize the total weighted completion time.
Constraints~(\ref{cons:NewTIF_01}) ensure each job is scheduled exactly once.
Constraints~(\ref{cons:NewTIF_02}) make sure that, at each time period $t$, there is only one job family that can be handled.
Constraints~(\ref{cons:NewTIF_03}) verify whether each batch respects the machine capacity.

A variant TIFV is adapted from~\cite{unal2019} and will also be tested in Section~\ref{sec:Computational results}; it replaces the constraints~(\ref{cons:NewTIF_03}) of TIF by
\begin{align}
&\sum_{i\in N} \sum_{\tau=\max\{t-p_i+1,1\}}^{\min\{t, |H_{\phi_i}|\}} x_{i\tau}v_i \leq V,
&&\forall t \in H,\label{cons:NewTIF_08}\\
&y_{jt} \geq x_{it},
&&\forall  i \in N_j, t\in H_j, j\in \text{$\cal{F}$}.\label{cons:NewTIF_09}
\end{align}
Constraints~(\ref{cons:NewTIF_08}) impose the machine capacity in each time period, while constraints~(\ref{cons:NewTIF_09}) ensure the relationship between family $j$ and its job $i$ at time period $t$.

\subsection{Set-partitioning formulation}
\label{subsec:Set-partitioning formulation}
We generate a set of batches (patterns) $S_j = \{1,\ldots,|S_j|\}$ for each family $j$, and $\cup_{j=1}^m S_j=S$. Each batch~$s$ consists of a subset $J_s \subseteq N_j$ of jobs for which $\sum_{i\in J_s} v_i \le V$. The value $ w_{st} = \sum_{i\in J_s} w_i(t+p_i-1)$ is the cost of batch $s\in S_j$ if it starts at time period $t$.

As before, we consider the time horizon $H=\{1,\ldots,H_{max}\}$. 
Each batch $s\in S_j$ of family $j$ has a set $H_s=\{1,\ldots,H_{max}-q_j+1\}$ of potential start periods.
Let variable $z_{st}$ be one if batch $s$ starts in time period $t$, and zero otherwise, for $t\in H_s$, $s \in S$.
\begin{align}
\label{obj:SPF}
\min \ &  \sum_{s \in S}\sum_{t \in H_s} w_{st} z_{st}   \\
s.t. \ & \sum_{j\in \mathcal{F}}\sum_{s \in S_j}\sum_{\tau =\max\{t-q_j+1,1\}  }^{\min\{t,|H_s|\}} z_{s\tau} \le 1, && \forall t \in H,\label{cons:SPF_01} \\
& \sum_{s \in S : i\in J_s} \sum_{t \in H_s} z_{st} = 1, &&\forall i \in N.\label{cons:SPF_02}
\end{align}
The objective is to minimize the total weighted completion time. Constraints~(\ref{cons:SPF_01}) ensure the machine executes at most one batch at each time period, and constraints~(\ref{cons:SPF_02}) schedule each job once.
This formulation can be seen as a Dantzig-Wolfe decomposition of the TIF\@.
Appendix~\ref{Appendix:batch generation} describes a method to generate all feasible batches for a family.


\section{Column generation}
\label{sec:Column generation}

This section considers a CG algorithm with the linear-programming (LP) relaxation of the SPF as a master problem (MP). We discuss each of the main ingredients
of the algorithm below.
\begin{description}
\item[The dual problem] We introduce continuous dual variables $u_t\leq 0$ for each $t\in H$ and $\pi_i$ for each $i\in N$, corresponding to constraints~(\ref{cons:SPF_01}) and~(\ref{cons:SPF_02}), respectively. The dual problem of the LP is then:
\begin{align}
\max \ &  \sum_{t\in H} u_t + \sum_{i\in N} \pi_i  \\
s.t. \ &  \sum_{\tau=t}^{t+q_j-1}u_{\tau} + \sum_{i \in J_s} \pi_i   \le w_{st}, && \forall s\in S_j, t\in H_s, j\in \mathcal{F}. \label{dualconstraint}
\end{align}
The reduced cost of a non-basic variable in the relaxed SPF is $w_{st}-\sum_{i \in J_s} \pi_i- \sum_{\tau=t}^{t+q_j-1}u_{\tau}$. If this value is less than zero, we add a new column.

\item[Decomposition and pricing] In the pricing problem, we check whether a new variable can be added to the MP with negative reduced cost by searching for a family~$j$ and period~$t$ for which the constraint~\eqref{dualconstraint} is violated the most.
This leads to a separate pricing problem for each family~$j$ and period~$t$ as follows:

\begin{align}
\label{obj:Pricing}
\min \ & (t+q_j-1) \sum_{i \in N_j} w_i x_i - \sum_{i \in N_j} \pi_i x_i - \sum_{\tau=t}^{t+q_j-1}u_{\tau}
\\
s.t. \ & \sum_{i \in N_j} x_i v_i \le V, \label{cons:Pricing01} \\
& x_i \in \{0,1\} \label{cons:Pricing02} ,
\end{align}
where binary variable $x_i=1$ if job $i$ is loaded into this batch, 0 otherwise.
The constraint~\eqref{cons:Pricing01} verifies the total machine capacity.  Consequently, each pricing problem for family~$j$ and time period~$t$ is a 0-1 knapsack problem with profit values $\pi_i - (t+q_j-1)  w_i$ for each item $i\in N_j$, which we can solve
using standard dynamic programming~(see \cite{wolsey1998}, Section 5.4.1).

\item[Column selection] In each CG round, many columns might be added (one for each period/family combination).
In order to control the formulation size, we define a parameter $Col\_number$: in each round
we only include the first $Col\_number$ columns with lowest reduced costs (or the total number, if it is less).

\item[Initial feasible solution] We use a successive knapsack (SK) heuristic from~\cite{dobson2001} to generate an initial feasible solution for the restricted MP (RMP)\@. The SK consists of two steps: first, for each family, successively solve knapsack problems for the maximization of job weights, until all jobs are batched (each constructed batch is iteratively removed from $N_j$). Second, sequence all  the batches according to the WSPT rule. Additionally, in order to guarantee feasibility in every step of the B\&P procedure to be described below, we add~$\vert N \vert$ \emph{super columns}: one for every constraint~\eqref{cons:SPF_02}. Each of these columns then has a large coefficient in the objective function, a coefficient equal to one for its associated row in constraints~\eqref{cons:SPF_02}, and a coefficient equal to zero for all other rows.
\end{description}
The overall scheme of the CG procedure is described in Algorithm~\ref{algorithm:Column genreration}.
\begin{algorithm}
	\caption{Column generation}
	\begin{algorithmic}[1]
		\State Generate initial columns for RMP
		\Repeat
		\State Solve RMP by an LP solver and obtain the dual solutions $\pi_i$, $i \in N$, and $u_t$,  $t\in H$
		\State Solve the pricing problems
		\State Add up to $Col\_number$ new columns
		\Until{no negative reduced cost is found}
	\end{algorithmic}
	\label{algorithm:Column genreration}
\end{algorithm}

Based on this CG routine, we develop a B\&P algorithm and a CG-based heuristic (CGH)\@. The CGH simply solves the MIP formulation with all the columns added for finding the LP optimum.
For the B\&P, at each node of the B\&B tree, we solve the LP relaxation of the RMP of the SPF and obtain an optimal solution $z^*$.
If $z^*$ is fractional, we need to branch at the corresponding node. For this purpose, we use a hybrid branching framework based on the method of~\cite{ryan1981}.
Firstly, we search for a pair of jobs $i$ and $i'$ with
\begin{equation}
0< \sum_{s \in S' : i,i'\in J_s} \sum_{t \in H_s} z_{st} < 1,
\label{Ryan_Foster}
\end{equation}
where $S'$ is the set of current batches (columns) in the RMP\@.
Then we can create two new subproblems (child nodes) in the B\&B tree. In the first subproblem, we only consider columns with $i$ and $i'$ in the same batch. In the second subproblem, jobs $i$ and $i'$ are assigned to different batches. As a result, the pricing problem will become a knapsack problem with conflicts, which we solve using the B\&B algorithm of~\cite{sadykov2013}. 
When family $j$ has no conflicts (yet) among its jobs, we solve the pricing problem by dynamic programming.

If there is no pair of jobs that satisfies equation~\eqref{Ryan_Foster}, we branch directly on a variable $z_{st}$ with a fractional value, and two subproblems with $z_{st}=1$ and $z_{st}=0$ are then created. In the subproblem with $z_{st}=1$ and its child nodes, we will not consider jobs $i\in J_s$ anymore, because they are already decided to be in batch $s$, and the time periods $t$ to $t+q_j-1$ are set to be occupied by batch $s$, where $q_j$ is the processing time of family $j$ if $s\in S_j$, so no further pricing problems are needed for those periods.
For the subproblem with $z_{st}=0$, 
we need to prevent the generation of the same column at time period $t$ for family $j$ if $s\in S_j$.
Consequently, for pricing problems of the subproblem with $z_{st}=0$ and its child nodes, we use a solver to solve the model \eqref{obj:Pricing}--\eqref{cons:Pricing02} with the following additional constraints:
\begin{align}
& \sum_{i\in J_s}x_i \leq |s|-1+(1-y)M,  \\
&\sum_{i \notin J_s}x_i \geq 1-yM,\\
& y \in \{0,1\} ,
\end{align}
at time period $t$ for family $j$.
Variables $y$ is a binary variable and $M$ is a large number.

Example~\ref{example:branching} below shows an instance for which, in the context of our problem, the branching rule of~\cite{ryan1981} fails to find a pair of jobs to branch on.  That is, there might be a fractional solution for the LP relaxation of SPF for which there does not exist a pair of jobs~$i$ and~$i^\prime$ for which inequalities~\eqref{Ryan_Foster}
are satisfied.
The reason is that our SPF might produce duplicate columns that start at different time periods. A related result about the branching rule of~\cite{ryan1981} can be found in~\cite{vance1994}, who study a cutting stock problem.

\begin{example}
	\label{example:branching}
	\normalfont
	Consider an instance with two jobs $i$ and $i'$ with  $w_i=w_{i'}=1$ and $p_i=p_{i'}=1$	, which are from the same family.
	Let $s$ and $s'$ be the batches that contain job $i$ and $i'$, respectively.	
	Let $z^*$ be a fractional optimal solution of the LP relaxed SPF,
	and assume there are only non-zero variables $z_{s1}^*$, $z_{s2}^*$, $z_{s'1}^*$ and $z_{s'2}^*$, which all have the value 0.5.
	The variables $z_{s1}$ and $z_{s'1}$ pertain to a start at time period~1, while the variables $z_{s2}$ and $z_{s'2}$ start at time period 2. The detailed data and Gantt chart are shown in Table~\ref{table:branching} and Figure~\ref{fig:branching}, respectively. One can easily verify that inequalities~\eqref{Ryan_Foster} are not satisfied for jobs~$i$ and~$i'$, and that, despite the fractional solution, the branching rule does not provide a pair of jobs to branch on.
\end{example}
 \begin{figure}[t]
	\begin{minipage}{0.5\linewidth}
		\centering
		\captionof{table}{Data for Example~\ref{example:branching}}
	    \begin{tabular}{cccc}
	    \toprule[1pt]
		\multicolumn{1}{l}{variable} & \multicolumn{1}{l}{jobs } & \multicolumn{1}{l}{weight} & \multicolumn{1}{l}{processing time} \\
		\midrule
		$z_{s1}$      & $i$     & 1     & 1 \\
		$z_{s'1}$     & $i'$    & 1     & 1 \\
		$z_{s2}$      & $i$     & 1     & 1 \\
		$z_{s'2}$     & $i'$    & 1     & 1 \\
		\bottomrule[1pt]
	\end{tabular}
	\label{table:branching}
	\end{minipage}
	\hfill
	\begin{minipage}{0.5\linewidth}
	\pgfmathsetlengthmacro\timeunit{3*6pt}
	\pgfmathsetlengthmacro\vertunit{18.18034pt}
	\centering
	\begin{tikzpicture}
	\node at (0.6*\timeunit,1.4) [rectangle, anchor=south west, minimum height = \vertunit, draw, minimum width = 3*\timeunit, fill=white!10]{$z_{s1}=0.5$};
	
	\node at (3.6*\timeunit,1.4) [rectangle, anchor=south west, minimum height = \vertunit, draw, minimum width = 3*\timeunit, fill=white!10]{$z_{s2}=0.5$};
	
	\node at (0.6*\timeunit,0.5) [rectangle, anchor=south west, minimum height = \vertunit, draw, minimum width = 3*\timeunit, fill=white!10]{$z_{s'1}=0.5$};
	
	\node at (3.6*\timeunit,0.5) [rectangle, anchor=south west, minimum height = \vertunit, draw, minimum width = 3*\timeunit, fill=white!10]{$z_{s'2}=0.5$};
	
	\draw (0.6*\timeunit,0) -- (6.6*\timeunit,0);
	
	\draw (0.6*\timeunit,0) -- (0.6*\timeunit,0.2);
	\draw (3.6*\timeunit,0) -- (3.6*\timeunit,0.2);
	\draw (6.6*\timeunit,0) -- (6.6*\timeunit,0.2);
	
	\node at (2.1*\timeunit,-0.4) {1};
	\node at (5.1*\timeunit,-0.4) {2};
	\end{tikzpicture}
	\caption{Gantt chart for Example~\ref{example:branching}}
	\label{fig:branching}
\end{minipage}
\end{figure}

To choose a pair of jobs for branching, we use the method of~\cite{held2012} as follows. For each pair of jobs $i$ and $i'$, define
\begin{equation}
q(i, i')=\frac{\sum_{s \in S' : i,i'\in J_s}\sum_{t \in H_s}z_{st}}{\frac{1}{2}(\sum_{s \in S' : i\in J_s}\sum_{t \in H_s}z_{st}+\sum_{s \in S' : i'\in J_s}\sum_{t \in H_s}z_{st})}.
\end{equation}
If $q(i, i')$ is close to zero, jobs $i$ and $i'$ are more likely to be batched separately in this fractional solution, while they are mostly batched together when $q(i, i')$ is close to one.
In any one of these two cases, one child node's lower bound will be similar to the parent node's lower bound. Thus, we choose a job pair $i$ and $i'$ for which $q(i, i')$ is closest to 0.5. We first explore the child node with jobs $i$ and $i'$ together, since this increases the probability of obtaining an integer solution. When branching directly on a variable, we choose the variable with its value closest to 0.5, and we explore the child node with $z_{st}=1$ firstly.
The B\&B tree is explored in a depth-first manner.


\section{Preprocessing}
\label{sec:Preprocessing}
This section introduces a preprocessing procedure to decrease the formulation size.
For each family~$j$, at most $|N_j|$ batches are formed.  Obviously, more batches lead to a potentially longer schedule, which increases the formulation size. We therefore set up a preprocessing procedure to obtain an upper bound $B_j\le |N_j|$ on  the required number of batches for each family $j$.

Firstly, for family $j\in \mathcal{F}$,  we order the jobs in non-increasing order of their sizes, and iteratively load them until the remaining machine capacity is not sufficient for the next job. Let $N_j^b$ stand for the number of jobs in the loaded batch. We observe that any subset of $N_j^b$ jobs from set $N_j$ will constitute a batch that respects the machine capacity.
Next, let $B_j = \lceil|N_j|/N_j^b  \rceil$, where $\lceil \cdot \rceil$ is the ceiling function.

\begin{lemma}
In the problem $1|p\text{-}batch,v_i,incompatible|\sum w_iC_i$, there is at least one optimal solution in which the batch number for each family $j\in \mathcal{F}$ is less than or equal to $B_j$.
\end{lemma}

\begin{proof}
By contradiction.  Assume that in each optimal solution, there is at least one family $j$ with strictly more than $B_j$ batches. This family $j$ then has at least has two batches $k$ and $k'$ that contain strictly less than $N_j^b$ jobs. Let $k$ be batch with the earliest completion time. At least one job in batch $k'$ can be moved to batch $k$, which means that the total weighted completion time can be improved. This contradicts the hypothesis.
\end{proof}

In a preprocessing step, the value of $|\Omega|$ in ABF can be set to $\sum_{j\in \mathcal{F}}B_j$, and the right-hand side of constraints~(\ref{cons:Step_02}) can be replaced by $B_j$. The value $H_{max}$ in TIF and SPF can be chosen as  $\sum_{j\in \mathcal{F}}B_jq_j$.

\section{Proximity search}
\label{sec:Proximity search}

Proximity search was first proposed by~\cite{fischetti2014}.  It is a general framework for iteratively refining a given reference solution to an optimization problem. Consider a 0-1 integer program of the general form

\begin{align}
\label{obj:proximity search}
\min \ & \displaystyle c^Tx \notag
\\
s.t.\ &Ax\geq b,\notag  \\
&x_j \in \{0,1\}, j \in \cal{B}, \notag
\end{align}

where $A$ is an $m \times n$ input matrix,  $b$ and $c$ are input vectors of dimension $m$ and $n$, and $\cal{B}$ is the index set of the binary variables. Proximity search improves a given reference solution $\tilde{x}$ iteratively. At each iteration, to improve $\tilde{x}$, an explicit cutoff constraint
\begin{equation}
\label{proximity:hard_cut off}
c^Tx \leq c^T\tilde{x}-\theta
\end{equation}
is added to the original formulation, where $\theta>0$ is a given tolerance that ensures a minimum improvement in the objective. The original objective function is replaced by the proximity function
\begin{equation}
\label{proximity:hard_obj}
\min \displaystyle \triangle (x,\tilde{x}) = \sum_{\tilde{x}_j=0} x_j + \sum_{\tilde{x}_j=1} (1-x_j),
\end{equation}
which is to minimize the Hamming distance between $x$ and $\tilde{x}$. A generic solver can work as a black box to solve the thus-modified formulation in the hope of improving the reference solution within a short Hamming distance.

In this work we apply the version from~\cite{fischetti2014} called ``proximity search with an incumbent".
When providing a reference solution $\tilde{x}$, the cutoff constraint may lead to infeasibility if the tolerance $\theta$ is too aggressive. \cite{fischetti2014} therefore introduce a continuous variable $\omega \geq 0$ and weaken the cutoff constraint~(\ref{proximity:hard_cut off}) to a soft version:
\begin{equation}
\label{proximity:soft_cut off}
c^Tx \leq c^T\tilde{x}-\theta(1-\omega),
\end{equation}
while minimizing
\begin{equation}
\label{proximity:soft_obj}
\triangle (x,\tilde{x}) + M\omega
\end{equation}
instead of objective~(\ref{proximity:hard_obj}), where $M$ is a positive number that is sufficiently large compared to $\triangle (x,\tilde{x})$.
We apply the framework provided by proximity search to the TIF, and refer to it as PS\@. The SK heuristic is used to generate an initial reference solution.
Algorithm~\ref{algorithm:proximity search} gives a basic structure that we implement; more details can be found in~\cite{fischetti2014}.

\begin{algorithm}[h]
	\caption{Proximity search (PS)}
	\begin{algorithmic}[1]
		\State Generate a reference solution $\tilde{x}$ by SK
		\Repeat
			\State add cutoff constraint~(\ref{proximity:soft_cut off}) to the model
			\State Replace original objective function by~(\ref{proximity:soft_obj}) to be minimized
			\State Run the modified model on the solver until a termination criterion is reached; let $x^*$ be the best feasible solution found
			\If{$c^Tx^*<c^T\tilde{x}$}
			\State Recenter $\triangle (x,\cdot)$ by setting $\tilde{x}=x^*$
			\Else
			\State update $\theta$	
			\EndIf
		\Until{an overall termination criterion is reached}
	\end{algorithmic}
	\label{algorithm:proximity search}
\end{algorithm}

\section{Computational results}
\label{sec:Computational results}
This section reports on the computational performance of our proposed models and algorithms. We first introduce the experimental setup in Section~\ref{subsec:Experimental setup}, and subsequently, Section~\ref{subsec:Exact methods and LB} presents the computational results for both the exact methods (i.e., ABF, TIF, TIFV, TIFM, SPF, and B\&P) as well as the lower bounds. Section~\ref{subsec:Heuristics}, finally, compares the heuristic performance of CGH, PS, and a truncated B\&P (tB\&P) with previous methods from the literature.

\subsection{Experimental setup}
\label{subsec:Experimental setup}
To assess the computational performance of our methods, we generate three sets of test instances. For each set, the processing times, job weights, and job sizes are drawn randomly from a discrete uniform distribution on a given range $[a, b]$. 
Table~\ref{table:Experimental settings} summarizes the parameter settings for each of the three instance sets.

\begin{table}[h]
	\centering
	\begin{threeparttable}
	\caption{Experimental settings used to generate the sets K2008, K2008u, and H2017 of test instances}
	\begin{tabular}{CCCC}
		\toprule[1pt]
		\text{Set}   & \text{K2008}&\text{K2008u} &\text{H2017} \\
		\midrule
		n     & 20, 40, 60, 80, 100    &120, 150       & 80, 100 \\
		m     & 2, 4, 6, 10  & 2, 4, 6, 10    & 3, 5 \\
		q_j     & [10j, 10j+10] & 1      &  [1, 15]  \\
		w_i     & [1, 10] & [1, 10]      & [1, 10] \\
		v_i   & [1, 10], [2, 4], [4, 8]  & [1, 10], [2, 4], [4, 8]       &  [1, 50] \\
		V     & 10  & 10      & 50 \\
		\bottomrule[1pt]
	\end{tabular}
	\label{table:Experimental settings}
	\end{threeparttable}
\end{table}

The first set of test instances, which we refer to as K2008, is generated analogously to the instances described in~\cite{kashan2008}. For $n \in \{20, 40, \ldots, 100\}$ and~$m \in \{2,4,6,10\}$, the number of jobs in a family equals $|N_j| = \lfloor n/m \rfloor$ for $j \in \{1,2,\ldots, m-1\}$ (with $\lfloor \cdot \rfloor$ the floor function), and $|N_m| = n-\sum_{j=1}^{m-1}|N_j|$. For each family~$j \in \mathcal{F}$ and job $i \in N$, a processing time $q_j$ and a job weight~$w_i$ are drawn randomly from the range $[10j, 10j+10]$ and $[1, 10]$, respectively. Depending on the experimental setting, we draw a job size $v_i$ for job $i \in N$ from the interval $[1, 10]$, $[2, 4]$, or $[4, 8]$. The machine capacity equals $V=10$. For each of the parameter settings, ten instances are generated, leading to a total of 600 instances.
The second set of test instances, which we refer to as K2008u, is generated in exactly the same way as the K2008 set, except that~$n \in \{120, 150\}$ and all families~$j \in \mathcal{F}$ have unit processing times~$q_j = 1$.
The third set of test instances, finally, is referred to as H2017, and is generated analogously to the instances of~\cite{ham2017}. It contains instances with $n \in \{80, 100\}$ and $m \in \{3,5\}$, and the number of jobs $|N_j|$ of each family $j \in \mathcal{F}$ is obtained in the same way as for the K2008 instances. The processing times, job weights and sizes are generated from ranges $[1,15]$, $[1,10]$, and $[1, 50]$, respectively. The machine capacity equals $V=50$. Again, ten instances are generated for each of the experimental settings, leading to a total of~$40$ instances.

All our algorithms are implemented in the C++ programming language, and we use Gurobi 9.0.2 as the LP and MILP solver. The experiments are performed on a computer with an Intel i7-4790 @ 3.60GHz processor using a single thread and 32GB of RAM.



\subsection{Exact methods and lower bounds}
\label{subsec:Exact methods and LB}

In this subsection, we first look into the tests of ABF, TIF, TIFV, TIFM, SPF, and B\&P on instances with $n=20$ from K2008; the results are depicted in Table~\ref{table:Non-unit_20}.  Based on this table, we choose TIF, SPF and B\&P, which perform better, to test further instances from the set K2008 with $n\in\{40,60,80\}$, and the set K2008u. Finally, we compare the lower bound obtained by the LP relaxation of SPF with previous ones from the literature.

In all these experiments, we set a time limit of 1200 seconds. For SPF, the computational time includes the time for batch generation.
We set the number $Col\_number$ for B\&P as $20\times m$ and $10\times m$ at the root node and non-root nodes, respectively.
To guarantee a fair comparison with B\&P, we use the SK heuristic, which requires only very limited computational effort, to obtain an initial solution for each method.
The preprocessing from Section~\ref{sec:Preprocessing} is applied for each method.
In all tables of this subsection, the average CPU time and the node number for B\&P are only calculated with the instances solved to guaranteed optimality.

Table~\ref{table:Non-unit_20} reports on the computational performance of all the discussed exact solution methods for the K2008 instances with~$n = 20$. For each method, the table contains the average CPU time in seconds as well as the number of instances solved within the time limit.  We find that:
\begin{itemize}
	\setlength\itemsep{0.1cm}
	\item SPF and B\&P display the best performance since these methods can solve all K2008 instances with $n = 20$ to guaranteed optimality within the time limit. Moreover, B\&P requires significantly less time to solve these small instances.
	\item ABF can solve the least instances of all solution methods. Its performance is especially poor when the job size~$v_i\in [4, 8]$.
	\item Among the time-indexed formulations, TIF outperforms TIFV and TIFM.\@ One possible explanation for this is that, for TIFV, constraints~(\ref{cons:NewTIF_09}) increase the model size and slow down the computation time. The relatively poor performance of TIFM, in turn, can be explained by the big-M constraints.
\end{itemize}

\begin{table}[t]
	\small
	\setlength{\tabcolsep}{4pt}	
	\centering	
	\begin{threeparttable}	
	\caption{Computational performance of exact solution methods for K2008 instances with $n = 20$; `time' indicates the average CPU times in seconds, `opt' the number of solved instances (out of 10) within the 1200-seconds time limit, and `nodes' the average number of nodes computed by B\&P}
	\begin{tabular}{ccRRRRRRRRRRRRR}
		\toprule[1pt]
		& & \multicolumn{2}{c}{ABF} & \multicolumn{2}{c}{TIF} & \multicolumn{2}{c}{TIFV} & \multicolumn{2}{c}{TIFM} & \multicolumn{2}{c}{SPF} & \multicolumn{3}{c}{B\&P} \\
		\cmidrule(lr){3-4}
		\cmidrule(lr){5-6}
		\cmidrule(lr){7-8}
		\cmidrule(lr){9-10}
		\cmidrule(lr){11-12}
		\cmidrule(lr){13-15}
		$m$ & $v_i$ & \multicolumn{1}{c}{time} & \multicolumn{1}{c}{opt} & \multicolumn{1}{c}{time} & \multicolumn{1}{c}{opt} & \multicolumn{1}{c}{time} & \multicolumn{1}{c}{opt} & \multicolumn{1}{c}{time} & \multicolumn{1}{c}{opt} & \multicolumn{1}{c}{time} & \multicolumn{1}{c}{opt} & \multicolumn{1}{c}{time} & \multicolumn{1}{c}{opt} & \multicolumn{1}{c}{nodes} \\
		\midrule
		
		2 & [1, 10] & 79.40 & 10    & 20.87 & 10    & 104.22 & 10    & 63.13 & 10    & 2.53  & 10    & \bf{0.79}  & \bf{10}    & 11.60 \\
		4 & [1, 10] & 257.02 & 6     & 35.55 & 10    & 171.68 & 10    & 410.64 & 9     & 1.98  & 10    & \bf{0.75}  & \bf{10}    & 2.80 \\
		6 & [1, 10] & 554.03 & 5     & 67.32 & 10    & 188.54 & 9     & 267.84 & 9     & 3.86  & 10    & \bf{1.25}  & \bf{10}    & 1.80 \\
		10 & [1, 10] & 204.09 & 7     & 23.76 & 10    & 336.64 & 10    & 361.13 & 9     & 4.34  & 10    & \bf{1.60}  & \bf{10}    & 1.00 \\ \midrule
		2 & [2, 4] & \bf{0.83}  & \bf{10}    & 7.13  & 10    & 21.61 & 10    & 240.06 & 6     & 3.80  & 10    & 14.66 & 10    & 258.20 \\
		4 & [2, 4] & 2.73  & 10    & 104.48 & 10    & 243.37 & 9     & 37.29 & 10    & \bf{2.54}  & \bf{10}    & 3.18  & 10    & 37.60 \\
		6 & [2, 4] & 4.53  & 10    & 39.89 & 9     & 102.20 & 9     & 11.61 & 10    & 2.23  & 10    & \bf{1.06}  & \bf{10}    & 4.40 \\
		10 & [2, 4] & 3.28  & 10    & 15.71 & 10    & 8.68  & 10    & \bf{0.61}  & \bf{10}    & 2.04  & 10    & 1.03  & 10    & 1.00 \\ \midrule
		2 & [4, 8] & 116.41 & 10    & 6.45  & 10    & 173.85 & 10    & 76.44 & 10    & 1.67  & 10    & \bf{0.71}  & \bf{10}    & 9.40 \\
		4 & [4, 8] & 768.34 & 4     & 20.21 & 9     & 289.73 & 6     & 187.82 & 10    & 1.56  & 10    & \bf{0.68}  & \bf{10}    & 1.60 \\
		6 & [4, 8] & \text{---} & 0     & 14.39 & 10    & 255.51 & 10    & 191.14 & 10    & 2.76  & 10    & \bf{1.14}  & \bf{10}    & 1.00 \\
		10 & [4, 8] & 108.10 & 2     & 35.78 & 10    & 468.22 & 4     & 335.19 & 3     & 4.87  & 10    & \bf{1.96}  & \bf{10}    & 1.00 \\
		\midrule
		\multicolumn{2}{l}{\text{average/total}} & 132.17 & 84    & 32.67 & 118   & 178.88 & 107   & 164.93 & 106   & 2.85  & 120   & \bf{2.40}  & \bf{120}   & 27.62 \\
		\bottomrule[1pt]
	\end{tabular}
	\label{table:Non-unit_20}
	\end{threeparttable}	
\end{table}

Given the superior performance of TIF, SPF, and B\&P, we focus only on these methods in the remainder of this subsection. Table~\ref{table:Overall} summarizes the outcome of these experiments on instances from K2008 ($n\in\{20,40,60,80\}$) and K2008u ($n\in\{120,150\}$), where the results are split up according to the number of jobs~$n$, the job size~$v_i$, and the number of families~$m$. For more detailed computational results, we refer to Tables~\ref{table:Non-unit}-\ref{table:Unit} in Appendix~\ref{Appendix:Experimental data}. We find that:

		\begin{table}[t]
		\centering
		\caption{Computational performance of TIF, SPF, and B\&P for K2008 and K2008u instances, split up according to number of jobs $n$, size of jobs $v_i$, and number of families $m$; `time' indicates the average CPU times in seconds, `opt' the number of solved instances (out of 120, 240 and 180, respectively) within the 1200-seconds time limit}
		\begin{threeparttable}
		\small
		\setlength{\tabcolsep}{4pt}	
		
				\begin{subtable}[b]{1\linewidth}
				
				\begin{tabular}{cRRRRRR}
					\toprule[1pt]
				 & \multicolumn{2}{c}{TIF} & \multicolumn{2}{c}{SPF} & \multicolumn{2}{c}{B\&P} \\
					\cmidrule(lr){2-3}
					\cmidrule(lr){4-5}
					\cmidrule(lr){6-7}
					\multicolumn{1}{c}{$n$}& \multicolumn{1}{c}{time} & \multicolumn{1}{c}{opt} & \multicolumn{1}{c}{time} & \multicolumn{1}{c}{opt } & \multicolumn{1}{c}{time } & \multicolumn{1}{c}{opt} \\
					\midrule
					\multicolumn{1}{c}{20} & 32.67 & 118   & 2.85  & \bf{120}   & 2.40  & \bf{120} \\
					\multicolumn{1}{c}{40} & 347.37 & 36    & 76.67 & \bf{120}   & 91.05 & 106 \\
					\multicolumn{1}{c}{60} & 589.49 & 8     & 376.45 & \bf{80}    & 301.74 & 65 \\
					\multicolumn{1}{c}{80} & \text{---} & 0     & 531.20 & \bf{24}    & 571.39 & 21 \\
					\multicolumn{1}{c}{120} & 442.85 & 20    & 85.14 & \bf{108}   & 235.27 & 32 \\
					\multicolumn{1}{c}{150} & 79.01 & 3     & 66.48 & \bf{85}    & 310.44 & 17 \\
					\midrule
					\text{total}   &       & 185   &       & \bf{537}   &       & 361 \\
					\bottomrule[1pt]
				\end{tabular}
				\caption{Number of jobs $n$}
				\label{table:Overall_n}
			\end{subtable}
			
		\begin{subtable}[b]{.45\linewidth}
			
			\begin{tabular}{cRRRRRR}
				\toprule[1pt]
				 & \multicolumn{2}{c}{TIF} & \multicolumn{2}{c}{SPF} & \multicolumn{2}{c}{B\&P} \\
				\cmidrule(lr){2-3}
				\cmidrule(lr){4-5}
				\cmidrule(lr){6-7}
				\multicolumn{1}{c}{$m$}& \multicolumn{1}{c}{time} & \multicolumn{1}{c}{opt} & \multicolumn{1}{c}{time} & \multicolumn{1}{c}{opt } & \multicolumn{1}{c}{time } & \multicolumn{1}{c}{opt} \\
				\midrule
				2     & 217.06 & 78    & 121.62 & \bf{115}   & 93.76 & 69 \\
				4     & 151.05 & 39    & 160.97 & \bf{137}   & 142.50 & 86 \\
				6     & 129.13 & 35    & 103.71 & \bf{134}   & 145.32 & 96 \\
				10    & 85.73 & 33    & 114.69 & \bf{151}   & 197.11 & 110 \\
				\midrule
				\text{total}   &       & 185   &       & \bf{537}   &       & 361 \\
				\bottomrule[1pt]
			\end{tabular}
			\caption{Number of families $m$}
			\label{table:Overall_m}
		\end{subtable}
		\hspace{2ex}
		\begin{subtable}[b]{.45\linewidth}
			\begin{tabular}{cRRRRRR}
				\toprule[1pt]
				& \multicolumn{2}{c}{TIF} & \multicolumn{2}{c}{SPF} & \multicolumn{2}{c}{B\&P} \\
				\cmidrule(lr){2-3}
				\cmidrule(lr){4-5}
				\cmidrule(lr){6-7}
				\multicolumn{1}{c}{$v_i$}& \multicolumn{1}{c}{time} & \multicolumn{1}{c}{opt} & \multicolumn{1}{c}{time} & \multicolumn{1}{c}{opt } & \multicolumn{1}{c}{time } & \multicolumn{1}{c}{opt} \\
				\midrule
				\addlinespace[0.15cm]
				$[1, 10]$ & 189.80 & 60    & 131.81 & \bf{187}   & 144.30 & 126 \\
				\addlinespace[0.15cm]
				$[2, 4]$ & 149.57 & 64    & 163.32 & \bf{137}   & 90.13 & 68 \\
				\addlinespace[0.16cm]
				$[4, 8]$ & 150.98 & 61    & 94.98 & \bf{213}   & 179.92 & 167 \\
				\midrule
				\text{total}   &       & 185   &       & \bf{537}   &       & 361 \\
				\bottomrule[1pt]
			\end{tabular}	
			
			\caption{Job size $v_i$}
			\label{table:Overall_s}
			
		\end{subtable}
		\end{threeparttable}
	\label{table:Overall}
\end{table}

\begin{itemize}
	\setlength\itemsep{0.1cm}
	\item SPF performs best overall, as it solves the most instances. As~$n$ increases and~$m$ decreases, however, the performance deteriorates, since the formulation size increases exponentially with the number of jobs in each family. For some instances, the large model size even results in a failure to build the model within the time and memory limits (see also Tables~\ref{table:Non-unit}-\ref{table:Unit} in Appendix~\ref{Appendix:Experimental data}).
	\item B\&P only seems to perform really well for instances with $n = 20$. Yet, it can be a good alternative optimal procedure when the
	SPF runs out of memory.
	\item
	The ``packability'' of jobs, i.e., the degree to which jobs can be batched within the machine capacity, impacts the performance of SPF and B\&P\@. One possible explanation for this is that a higher packability due to the job size (e.g., $v_i \in [2,4]$) implies that there are more possible combinations to form a batch. This, in turn, increases the formulation size, and thus harms  the performance.
	\item TIF manages to outperform SPF and B\&P for some specific settings with low~$m$ and
	easily packable job sizes~$v_i$  (see Tables~\ref{table:Non-unit} and~\ref{table:Unit} in Appendix~\ref{Appendix:Experimental data}), making it an interesting alternative solution method for some scenarios. 
\end{itemize}

 Figure~\ref{fig:Overall_plot} shows the number of K2008 and K2008u instances (out of 720) solved by TIF, SPF, and B\&P as a function of the allocated computation time (in seconds). The figure indicates that for all reported computation times, SPF outperforms B\&P, which in turn performs better than TIF\@. Moreover, for both SPF and B\&P, the number of solved instances increases significantly especially during the first 600 seconds, after which it seems to stagnate. For TIF, on the other hand, the increase occurs slower and more gradually.

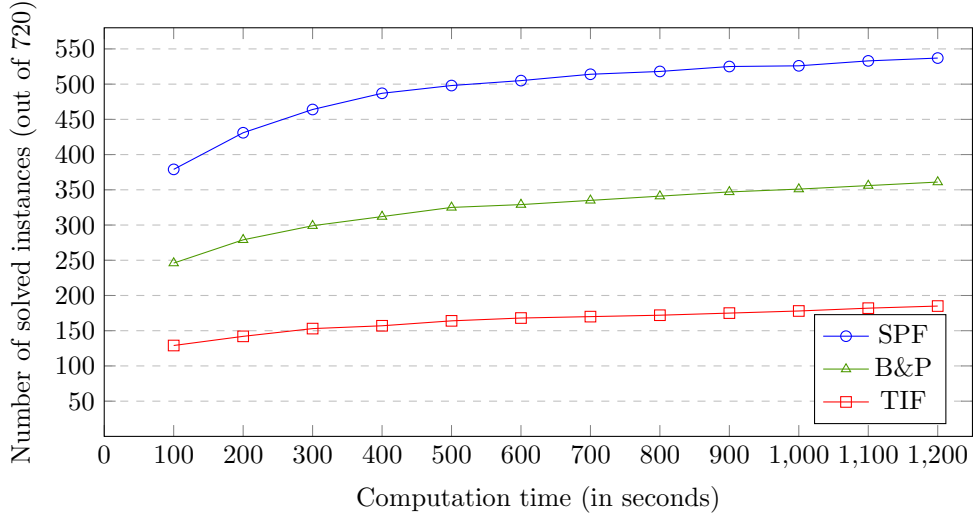
\begin{figure}[t]
	\centering
	\begin{tikzpicture}
		\begin{axis}[ 	
			height=7cm,
			width=13cm,
			xlabel={Computation time (in seconds)},
			ylabel={Number of solved instances (out of 720)},
			xmin=0, xmax=1250, 		
			ymin=0, ymax=580, xtick={0,100,200,300,400,500,600,700,800,900,1000,1100,1200},	
			ytick={50,100,150,200,250,300,350,400,450,500,550},	
			legend pos=south east,
			ymajorgrids=true,	grid style=dashed]
			
			\addplot[color=blue, mark=o,]			coordinates {		(100,379)	(200,431)	(300,464)	(400,487)	(500,498)	(600,505)	(700,514)	(800,518)	(900,525)	(1000,526) (1100,533) (1200,537)	};
			\addlegendentry{SPF}			
			\addplot[color=green1, mark=triangle,]			coordinates {		(100,246)	(200,279)	(300,299)	(400,312)	(500,325)	(600,329)	(700,335)	(800,341)	(900,347)	(1000,351) (1100,356) (1200,361)	};
			\addlegendentry{B\&P}
			\addplot[color=red,	mark=square,]		coordinates {		(100,129)	(200,142)	(300,153)	(400,157)	(500,164)	(600,168)	(700,170)	(800,172)	(900,175)	(1000,178) (1100,182) (1200,185)	};
			\addlegendentry{TIF}
			
		\end{axis}
	\end{tikzpicture}
	
	\caption{Number of K2008 and K2008u instances (out of 720) solved by TIF, SPF, and B\&P as a function of the allocated computation time}
	\label{fig:Overall_plot}
\end{figure}

Table~\ref{table:LB} reports on the quality of the different lower bounds for K2008 instances with $n \in \{20, 40\}$. Here, LBDN, LBAW, and LBKK refer to the lower bounds from \cite{dobson2001}, \cite{azizoglu2001}, and \cite{kashan2008}, respectively. The lower bound obtained by the LP relaxation of SPF is referred to as LBLP\@. For a given lower bound LB and an optimal solution value OPT, we measure the quality of the lower bound by the ratio $\text{LB}/\text{OPT} \times 100\%$. Hence, a higher value means a tighter bound. Table~\ref{table:LB} shows that 
 LBLP achieves the best quality, with a ratio that is on average around three to nine percentage points better than the other bounds.

\begin{table}[t]
	\centering
	\setlength{\tabcolsep}{3.5pt}	
	\small
	\begin{threeparttable}
	\caption{Quality of the lower bounds, measured by the ratio $\text{LB}/\text{OPT} \times 100\%$, for K2008 instances with $n \in \{20,40\}$}
	\begin{tabular}{ccRRRRRRRR}
		\toprule[1pt]
		& & \multicolumn{4}{c}{$n=20$} & \multicolumn{4}{c}{$n=40$} \\ \cmidrule(lr){3-6} \cmidrule(lr){7-10}
		$m$ & $v_i$
		&\multicolumn{1}{c}{LBDN}
		&\multicolumn{1}{c}{LBAW}
		&\multicolumn{1}{c}{LBKK}
		&\multicolumn{1}{c}{LBLP}
		&\multicolumn{1}{c}{LBDN}
		&\multicolumn{1}{c}{LBAW}
		&\multicolumn{1}{c}{LBKK}
		&\multicolumn{1}{c}{LBLP} \\
		\midrule
		2 & [1, 10] & 87.98\% & 82.02\%  & 89.79\% & \bf{98.69\%} & 93.23\% & 74.73\%  & 93.39\% & \bf{99.32\%}\\
		4 & [1, 10] & 88.39\% & 84.97\%  & 93.55\% & \bf{99.42\%}& 91.21\% & 77.24\%  & 91.83\% & \bf{98.47\%} \\
		6 & [1, 10] & 86.96\% & 87.24\%  & 94.99\% & \bf{99.47\%} & 88.97\% & 82.40\% & 91.54\% & \bf{97.90\%}\\
		10 & [1, 10] & 88.96\% & \bf{100\%}  & \bf{100\%} & \bf{100\%} & 89.65\% & 89.73\% & 95.38\% & \bf{99.08\%} \\ \midrule
		2 & [2, 4] & 96.33\% & 92.81\%  & 96.33\% & \bf{98.55\%} & 98.70\% & 89.20\% & 98.70\% & \bf{99.29\%} \\
		4 & [2, 4] & 94.75\% & 94.19\%  & 94.75\% & \bf{97.22\%} & 96.95\% & 91.85\%  & 96.95\% & \bf{98.18\%} \\
		6 &  [2, 4] & 93.50\% & \bf{99.01\%}  & 93.50\% & 98.25\% & 95.82\% & 92.42\%  & 95.82\% & \bf{96.86\%} \\
		10 & [2, 4] & \bf{100\%} & \bf{100\%}  & \bf{100\%} & \bf{100\%}& 93.83\% & \bf{98.60\%} & 93.83\% & 96.51\% \\ \midrule
		2 & [4, 8] & 86.88\% & 83.36\%  & 94.97\% & \bf{99.30\%} & 86.70\% & 77.94\%  & 96.91\% & \bf{99.67\%} \\
		4 & [4, 8] & 81.75\% & 81.67\%  & 98.45\% & \bf{99.19\%} & 83.91\% & 79.44\%  & 95.83\% &\bf{99.13\%} \\
		6 & [4, 8] & 79.62\% & 82.25\%  & 99.78\% & \bf{100\%} & 83.96\% & 80.75\%  & 96.13\% & \bf{99.40\%} \\
		10 & [4, 8] & 83.58\% & \bf{100\%} & \bf{100\%} &\bf{100\%} & 82.44\% & 83.97\%  & 98.06\% & \bf{98.47\%} \\
		\midrule
		\multicolumn{2}{c}{average} & 89.06\% & 90.63\%  & 96.34\% & \bf{99.17\%}  & 90.45\% & 84.86\%  & 95.36\% & \bf{98.52\%}\\
		\bottomrule[1pt]
	\end{tabular}
	\label{table:LB}
	\end{threeparttable}
\end{table}

\subsection{Heuristics}
\label{subsec:Heuristics}

In this subsection we first report the results for tB\&P, CGH, and PS on instances with $n\in \{80,100\}$ from set K2008\@. We compare with a hybrid ant colony framework called HACF2, which achieves a better overall performance than previous methods in~\cite{kashan2008}. Subsequently, we compare with a truncated constraint programming (CP) procedure on instances with $n\in \{80,100\}$ from set H2017; the CP model is taken from \cite{ham2017}, who apply CP to a generalization of BLSP for a problem from the semiconductor
industry. 

For tB\&P we use the same parameter settings as in Section~\ref{subsec:Exact methods and LB}, and for CGH we set $Col\_number$ equal to $50 \times m$. For PS, 
we use a tolerance $\theta = \alpha*c^T\tilde{x}$, where $\tilde{x}$ is the reference solution and, initially, $\alpha=0.02$. 
The initial reference solution for PS is produced by the SK heuristic. In each iteration of PS, if the incumbent solution does not improve within a time limit of 100 seconds, then we update $\alpha=\alpha \times 0.5$. We choose $M = 100,000$ in objective function~\eqref{proximity:soft_obj} and 
set the solver tolerance to $0.2 \times M$ (so the solver halts if the gap between upper and lower bound is less than this value).
 The CP model is solved using the CP Optimizer
of IBM ILOG CPLEX Optimization Studio 12.10, with 
one thread. For tB\&P, CGH and PS, the preprocessing described in Section~\ref{sec:Preprocessing} is applied.
Unless mentioned otherwise, we impose a time limit of 600 seconds.

Table~\ref{table:Heu} reports on the performance of CGH, tB\&P, PS, and HACF2 for K2008 instances with $n \in \{80,100\}$. Following~\cite{kashan2008}, we measure the performance of each heuristic by the ratio $\text{OBJ}/\text{LB}$, where OBJ is the objective value provided by the corresponding heuristics and  $\text{LB} = \max\{\text{LBKK},\text{LBAW}\}$. The choice of these latter lower bounds is made to ensure that our results are comparable with the ones of \cite{kashan2008}. Indeed, the results for HACF2 are taken directly from Tables 5 and 6 of~\cite{kashan2008}, while we test our own heuristics on instances generated by ourselves. In~\cite{kashan2008},  HACF2 is implemented in MATLAB 6.5.1 and run on a Pentium \rom{3} 800MHz computer with 128MB RAM using a stopping criterion of 40 generations. Since \cite{kashan2008} did not test instances with $m=10$, we report the computational results for these instances separately in Table~\ref{table:Heu_m=10}.

To compare our heuristics with CP, we apply our methods to the instances from H2017 with $n \in \{80, 100\}$. Table~\ref{table:Heu_CP} displays the results of this experiment, where we measure the performance of each method by the quantity $(\text{OBJ}-\text{LBLP})/\text{LBLP} \times 100\%$.

Tables~\ref{table:Heu}-\ref{table:Heu_CP} yield the following insights:
\begin{itemize}
	\setlength\itemsep{0.1cm}
	\item Table~\ref{table:Heu} shows that PS performs best on average. In particular, for those instances with a low number of families~$m$ or with a high packability in the job size~$v_i$, PS outperforms HACF2 both in terms of solution quality and computation time. As the number of families~$m$ increases, however, HACF2 regularly becomes better than PS.
	\item Table~\ref{table:Heu_CP} shows that PS and CP are the two best performing methods. One possible explanation of the better performance of PS when $m = 3$ is that the size of TIF, which forms the basis of PS, increases with the family number $m$.
\end{itemize}

\begin{table}
	\centering
	\setlength{\tabcolsep}{4pt}
	\begin{threeparttable}
	\caption{Ratio $\text{OBJ}/\text{LB}$ for the heuristics tB\&P, CGH, PS, and HACF2, and average CPU times in seconds for HACF2 (a 600-second time limit was used for all other heuristics) for K2008 instances with $n \in \{80, 100\}$ and $m \in \{2,4,6\}$}
		\begin{tabular}{ccRRRRRRRRRR}
			\toprule[1pt]
			& & \multicolumn{5}{c}{$n = 80$} & \multicolumn{5}{c}{$n = 100$} \\
			\cmidrule(lr){3-7} 			\cmidrule(lr){8-12}
			& & \multicolumn{1}{c}{tB\&P} & \multicolumn{1}{c}{CGH} & \multicolumn{1}{c}{PS} & \multicolumn{2}{c}{HACF2} & \multicolumn{1}{c}{tB\&P} & \multicolumn{1}{c}{CGH} & \multicolumn{1}{c}{PS} & \multicolumn{2}{c}{HACF2} \\
			\cmidrule(lr){3-3}
			\cmidrule(lr){4-4}
			\cmidrule(lr){5-5}
			\cmidrule(lr){6-7}
			\cmidrule(lr){8-8}
			\cmidrule(lr){9-9}
			\cmidrule(lr){10-10}
			\cmidrule(lr){11-12}
			$m$ & $v_i$ & \multicolumn{1}{c}{ratio} & \multicolumn{1}{c}{ratio} & \multicolumn{1}{c}{ratio} & \multicolumn{1}{c}{ratio} & \multicolumn{1}{c}{time} & \multicolumn{1}{c}{ratio} & \multicolumn{1}{c}{ratio} & \multicolumn{1}{c}{ratio} & \multicolumn{1}{c}{ratio} & \multicolumn{1}{c}{time} \\
			\midrule
			2 & [1, 10] & 1.0477 & 1.0508 & \bf{1.0444} & 1.0515 & 1514.2 & 1.0424 & 1.0505 & \bf{1.0396} & 1.0488 & 2862.0 \\
			4 & [1, 10] & \bf{1.0675} & 1.0765 & 1.0701 & 1.0684 & 657.7 & 1.0622 & 1.0699 & \bf{1.0611} & 1.0645 & 1358.3 \\
			6 & [1, 10] & \bf{1.0861} & 1.0956 & 1.0898 & 1.0887 & 454.7 & 1.0890 & 1.0965 & 1.0837 & \bf{1.0774} & 903.3 \\
			2 & [2, 4] & 1.0067 & 1.0067 & \bf{1.0060} & 1.0229 & 1219.5  & 1.0061 & 1.0061 & \bf{1.0049} & 1.0220 & 2396.3 \\
			4 & [2, 4] & 1.0203 & 1.0201 & \bf{1.0155} & 1.0260 & 571.0 & 1.0163 & 1.0163 & \bf{1.0139} & 1.0269 & 1160.4 \\
			6 & [2, 4] & 1.0325 & 1.0325 & \bf{1.0274} & 1.0330 & 403.0  & 1.0250 & 1.0257 & \bf{1.0199} & 1.0319 & 760.8 \\
			2 & [4, 8] & 1.0266 & 1.0268 & \bf{1.0258} & 1.0271 & 1063.4  & 1.0256 & 1.0275 & 1.0261 & \bf{1.0207} & 2003.0 \\
			4 & [4, 8] & 1.0378 & 1.0391 & 1.0410 & \bf{1.0348} & 524.5 & 1.0346 & 1.0360 & 1.0363 & \bf{1.0343} & 993.1 \\
			6 & [4, 8] & \bf{1.0483} & 1.0501 & 1.0499 & 1.0522 & 391.5 & \bf{1.0418} & 1.0505 & 1.0439 & 1.0562 & 716.0 \\
			\midrule
			\multicolumn{2}{l}{average} & 1.0415 & 1.0442 & \bf{1.0411} & 1.0450 & 755.5 & 1.0381 & 1.0421 & \bf{1.0366} & 1.0425 & 1461.5 \\
			\bottomrule[1pt]
		\end{tabular}
	\label{table:Heu}
	\end{threeparttable}
\end{table}

\begin{table}
	\centering
	\begin{threeparttable}
	\caption{Ratio $\text{OBJ}/\text{LB}$ for the heuristics tB\&P, CGH, and PS for K2008 instances with $n \in \{80, 100\}$ and $m = 10$}
	\begin{tabular}{cRRRRRR}
		\toprule[1pt]
		& \multicolumn{3}{c}{$n = 80$} & \multicolumn{3}{c}{$n = 100$}\\
		\cmidrule(lr){2-4}
		\cmidrule(lr){5-7}
		$v_i$ & \multicolumn{1}{c}{tB\&P} & \multicolumn{1}{c}{CGH} & \multicolumn{1}{c}{PS} & \multicolumn{1}{c}{tB\&P} & \multicolumn{1}{c}{CGH} & \multicolumn{1}{c}{PS}  \\
		\midrule
		$[1, 10]$ & 1.1070 & 1.1020 & \bf{1.1007} & 1.1021 & \bf{1.0987} & 1.0989 \\
		$[2, 4]$ & 1.0508 & 1.0498 & \bf{1.0465} & 1.0449 & 1.0445 & \bf{1.0409} \\
		$[4, 8]$ & \bf{1.0332} & 1.0342 & 1.0396 & 1.0521 & \bf{1.0497} & 1.0584 \\
		\midrule
		\text{average} & 1.0637 & \bf{1.0620} & 1.0622& 1.0664&	\bf{1.0643}	&1.0661
		 \\		
	\bottomrule[1pt]
	\end{tabular}
	\label{table:Heu_m=10}
	\end{threeparttable}
\end{table}%

\begin{table}
	\centering
	\begin{threeparttable}
	\caption{Ratio $(\text{OBJ}-\text{LBLP})/\text{LBLP} \times 100\%$ for tB\&P, CGH, PS, and CP  for  H2017 with $n \in \{80, 100\}$}
	\begin{tabular}{ccRRRRRRRR}
		\toprule[1pt]
		& & \multicolumn{4}{c}{$n = 80$} & \multicolumn{4}{c}{$n = 100$}\\
		\cmidrule(lr){3-6}
		\cmidrule(lr){7-10}
		$m$ & $v_i$ & \multicolumn{1}{c}{tB\&P} & \multicolumn{1}{c}{CGH} & \multicolumn{1}{c}{PS} & \multicolumn{1}{c}{CP} & \multicolumn{1}{c}{tB\&P} & \multicolumn{1}{c}{CGH} & \multicolumn{1}{c}{PS} & \multicolumn{1}{c}{CP} \\
		\midrule
		3 & [1, 50] & 1.54\% & 1.65\% & \bf{0.95\%} & 1.22\% & 1.12\% & 0.91\% & \bf{0.80\%} & 1.16\% \\
		5 & [1, 50] & 1.43\% & 2.01\% & 1.37\% & \bf{1.28\%} & 1.21\% & 1.48\% & 1.44\% & \bf{1.09\%} \\
		\midrule
		\text{average} && 1.49\% & 1.83\% & \bf{1.16\%} & 1.25\% &1.16\% &1.19\% &\bf{1.12\%} &1.13\%\\
		\bottomrule[1pt]
	\end{tabular}%
	\label{table:Heu_CP}
	\end{threeparttable}
\end{table}

\section{Conclusions}
\label{sec:Conclusions}

In this paper, we present new models and algorithms for the BLSP, which allow us to solve medium-sized and large instances to optimality. Specifically, the proposed SPF performs best, while the B\&P cannot equalize the performance of SPF\@. One possible direction for future research is to explore how one can improve the performance of B\&P by examining different branching strategies or by introducing heuristics to solve the pricing problem more efficiently.
The TIF and SPF can also be extended to other scheduling problems of batch processing machines with incompatible families and other objectives (e.g., total weighted tardiness). 
We also find that, on average, our proximity search heuristic performs better than previously proposed heuristics from the literature.
We hope that the findings of this paper can contribute to increasing the popularity of and progress in the scheduling of batch processing machines.
\section*{Acknowledgments}
Fan Yang is funded  by the China Scholarship Council.
Ben Hermans is funded by a Postdoctoral Fellowship of the Research Foundation -- Flanders (Fonds Wetenschappelijk Onderzoek). This work is partly supported by the National Natural Science Foundation of China 71801013 and the Beijing Social Science Foundation 18GLC078.

\bibliographystyle{apa}\label{sec:Reference}

\begin{thebibliography}{}
	
	\bibitem[\protect\astroncite{Azizoglu and Webster}{2001}]{azizoglu2001}
	Azizoglu, M. and Webster, S. (2001).
	\newblock Scheduling a batch processing machine with incompatible job families.
	\newblock {\em Computers \& Industrial Engineering}, 39(3-4):325--335.
	
	\bibitem[\protect\astroncite{Cakici et~al.}{2013}]{cakici2013}
	Cakici, E., Mason, S.~J., Fowler, J.~W., and Geismar, H.~N. (2013).
	\newblock Batch scheduling on parallel machines with dynamic job arrivals and
	incompatible job families.
	\newblock {\em International Journal of Production Research}, 51(8):2462--2477.
	
	\bibitem[\protect\astroncite{Coffman et~al.}{1978}]{coffman1978}
	Coffman, Jr., E.~G., Garey, M.~R., and Johnson, D.~S. (1978).
	\newblock An application of bin-packing to multiprocessor scheduling.
	\newblock {\em SIAM Journal on Computing}, 7(1):1--17.
	
	\bibitem[\protect\astroncite{Delorme et~al.}{2016}]{delorme2016}
	Delorme, M., Iori, M., and Martello, S. (2016).
	\newblock Bin packing and cutting stock problems: Mathematical models and exact
	algorithms.
	\newblock {\em European Journal of Operational Research}, 255(1):1--20.
	
	\bibitem[\protect\astroncite{Dobson and Nambimadom}{2001}]{dobson2001}
	Dobson, G. and Nambimadom, R.~S. (2001).
	\newblock The batch loading and scheduling problem.
	\newblock {\em Operations Research}, 49(1):52--65.
	
	\bibitem[\protect\astroncite{Fischetti and Monaci}{2014}]{fischetti2014}
	Fischetti, M. and Monaci, M. (2014).
	\newblock Proximity search for 0-1 mixed-integer convex programming.
	\newblock {\em Journal of Heuristics}, 20:709--731.
	
	\bibitem[\protect\astroncite{Graham et~al.}{1979}]{graham1979}
	Graham, R.~L., Lawler, E.~L., Lenstra, J.~K., and Rinnooy~Kan, A. H.~G. (1979).
	\newblock Optimization and approximation in deterministic sequencing and
	scheduling: a {S}urvey.
	\newblock {\em Annals of Discrete Mathematics}, 5:287--326.
	
	\bibitem[\protect\astroncite{Ham et~al.}{2017}]{ham2017}
	Ham, A., Fowler, J.~W., and Cakici, E. (2017).
	\newblock Constraint programming approach for scheduling jobs with release
	times, non-identical sizes, and incompatible families on parallel batching
	machines.
	\newblock {\em IEEE Transactions on Semiconductor Manufacturing},
	30(4):500--507.
	
	\bibitem[\protect\astroncite{Held et~al.}{2012}]{held2012}
	Held, S., Cook, W., and Sewell, E.~C. (2012).
	\newblock Maximum-weight stable sets and safe lower bounds for graph coloring.
	\newblock {\em Mathematical Programming Computation}, 4:363--381.
	
	\bibitem[\protect\astroncite{Jia et~al.}{2016}]{jia2016}
	Jia, Z., Wang, C., and Leung, J. Y.-T. (2016).
	\newblock An {ACO} algorithm for makespan minimization in parallel batch
	machines with non-identical job sizes and incompatible job families.
	\newblock {\em Applied Soft Computing}, 38:395--404.
	
	\bibitem[\protect\astroncite{Kashan and Karimi}{2008}]{kashan2008}
	Kashan, A.~H. and Karimi, B. (2008).
	\newblock Scheduling a single batch-processing machine with arbitrary job sizes
	and incompatible job families: An ant colony framework.
	\newblock {\em Journal of the Operational Research Society}, 59(9):1269--1280.
	
	\bibitem[\protect\astroncite{Kempf et~al.}{1998}]{kempf1998}
	Kempf, K.~G., Uzsoy, R., and Wang, C.-S. (1998).
	\newblock Scheduling a single batch processing machine with secondary resource
	constraints.
	\newblock {\em Journal of Manufacturing Systems}, 17(1):37--51.
	
	\bibitem[\protect\astroncite{Koh et~al.}{2004}]{koh2004}
	Koh, S.-G., Koo, P.-H., Ha, J.-W., and Lee, W.-S. (2004).
	\newblock Scheduling parallel batch processing machines with arbitrary job
	sizes and incompatible job families.
	\newblock {\em International Journal of Production Research},
	42(19):4091--4107.
	
	\bibitem[\protect\astroncite{Koh et~al.}{2005}]{koh2005}
	Koh, S.-G., Koo, P.-H., Kim, D.-C., and Hur, W.-S. (2005).
	\newblock Scheduling a single batch processing machine with arbitrary job sizes
	and incompatible job families.
	\newblock {\em International Journal of Production Economics}, 98(1):81--96.
	
	\bibitem[\protect\astroncite{Kucukkoc}{2019}]{kucukkoc2019}
	Kucukkoc, I. (2019).
	\newblock {MILP} models to minimise makespan in additive manufacturing machine
	scheduling problems.
	\newblock {\em Computers \& Operations Research}, 105:58--67.
	
	\bibitem[\protect\astroncite{Mathirajan and Sivakumar}{2006}]{mathirajan2006}
	Mathirajan, M. and Sivakumar, A.~I. (2006).
	\newblock A literature review, classification and simple meta-analysis on
	scheduling of batch processors in semiconductor.
	\newblock {\em The International Journal of Advanced Manufacturing Technology},
	29:990--1001.
	
	\bibitem[\protect\astroncite{M{\"o}nch et~al.}{2011}]{monch2011}
	M{\"o}nch, L., Fowler, J.~W., Dauzere-Peres, S., Mason, S.~J., and Rose, O.
	(2011).
	\newblock A survey of problems, solution techniques, and future challenges in
	scheduling semiconductor manufacturing operations.
	\newblock {\em Journal of Scheduling}, 14:583--599.
	
	\bibitem[\protect\astroncite{Nip et~al.}{2015}]{nip2015}
	Nip, K., Wang, Z., Talla~Nobibon, F., and Leus, R. (2015).
	\newblock A combination of flow shop scheduling and the shortest path problem.
	\newblock {\em Journal of Combinatorial Optimization}, 29:36--52.
	
	\bibitem[\protect\astroncite{Ozturk et~al.}{2014}]{ozturk2014}
	Ozturk, O., Begen, M.~A., and Zaric, G.~S. (2014).
	\newblock A branch and bound based heuristic for makespan minimization of
	washing operations in hospital sterilization services.
	\newblock {\em European Journal of Operational Research}, 239(1):214--226.
	
	\bibitem[\protect\astroncite{Pei et~al.}{2019}]{pei2019}
	Pei, J., Cheng, B., Liu, X., Pardalos, P.~M., and Kong, M. (2019).
	\newblock Single-machine and parallel-machine serial-batching scheduling
	problems with position-based learning effect and linear setup time.
	\newblock {\em Annals of Operations Research}, 272:217--241.
	
	\bibitem[\protect\astroncite{Pinedo}{2016}]{pinedo2016}
	Pinedo, M. (2016).
	\newblock {\em Scheduling: Theory, Algorithms, and Systems}.
	\newblock Springer, Cham.
	
	\bibitem[\protect\astroncite{Potts and Kovalyov}{2000}]{potts2000}
	Potts, C.~N. and Kovalyov, M.~Y. (2000).
	\newblock Scheduling with batching: A review.
	\newblock {\em European Journal of Operational Research}, 120(2):228--249.
	
	\bibitem[\protect\astroncite{Ryan and Foster}{1981}]{ryan1981}
	Ryan, D.~M. and Foster, B.~A. (1981).
	\newblock An integer programming approach to scheduling.
	\newblock In Wren, A., editor, {\em Computer Scheduling of Public Transport
		Urban Passenger Vehicle and Crew Scheduling}, pages 269--280, North Holland,
	Amsterdam.
	
	\bibitem[\protect\astroncite{Sadykov and Vanderbeck}{2013}]{sadykov2013}
	Sadykov, R. and Vanderbeck, F. (2013).
	\newblock Bin packing with conflicts: a generic branch-and-price algorithm.
	\newblock {\em INFORMS Journal on Computing}, 25(2):244--255.
	
	\bibitem[\protect\astroncite{Shen and Buscher}{2012}]{shen2012}
	Shen, L. and Buscher, U. (2012).
	\newblock Solving the serial batching problem in job shop manufacturing
	systems.
	\newblock {\em European Journal of Operational Research}, 221(1):14--26.
	
	\bibitem[\protect\astroncite{Smith}{1956}]{smith1956various}
	Smith, W.~E. (1956).
	\newblock Various optimizers for single-stage production.
	\newblock {\em Naval Research Logistics Quarterly}, 3:59--66.
	
	\bibitem[\protect\astroncite{Stork}{2001}]{stork2001}
	Stork, F. (2001).
	\newblock {\em Stochastic resource-constrained project scheduling}.
	\newblock Doctoral thesis, Technische Universit{\"a}t Berlin, Fakult{\"a}t II -
	Mathematik und Naturwissenschaften, Berlin.
	
	\bibitem[\protect\astroncite{Tamer~{\"U}nal et~al.}{2020}]{unal2019}
	Tamer~{\"U}nal, A., A{\u{g}}ral{\i}, S., and Caner~Ta{\c{s}}k{\i}n, Z. (2020).
	\newblock A strong integer programming formulation for hybrid flowshop
	scheduling.
	\newblock {\em Journal of the Operational Research Society}, 71(12):2042--2052.
	
	\bibitem[\protect\astroncite{Uzsoy}{1994}]{uzsoy1994}
	Uzsoy, R. (1994).
	\newblock Scheduling a single batch processing machine with non-identical job
	sizes.
	\newblock {\em International Journal of Production Research}, 32(7):1615--1635.
	
	\bibitem[\protect\astroncite{Vance et~al.}{1994}]{vance1994}
	Vance, P.~H., Barnhart, C., Johnson, E.~L., and Nemhauser, G.~L. (1994).
	\newblock Solving binary cutting stock problems by column generation and
	branch-and-bound.
	\newblock {\em Computational Optimization and Applications}, 3:111--130.
	
	\bibitem[\protect\astroncite{Wang and Cui}{2012}]{wang2012}
	Wang, Z. and Cui, Z. (2012).
	\newblock Combination of parallel machine scheduling and vertex cover.
	\newblock {\em Theoretical Computer Science}, 460:10--15.
	
	\bibitem[\protect\astroncite{Wolsey}{1998}]{wolsey1998}
	Wolsey, L.~A. (1998).
	\newblock {\em Integer Programming}.
	\newblock Wiley, New York.
	
\end{thebibliography}


\begin{appendices}

\section{A time-indexed formulation with big-M constraints}
\label{Appendix:TIF2017}

We describe an adapted  time-indexed formulation with `big-M' type constraints for the BLSP based on~\cite{cakici2013} and~\cite{ham2017}.
We use the definitions of $H$, $H_j$ and $\Omega$ as before.
Let $\Gamma_j$ be the batch set for family $j\in \mathcal{F}$, containing $|N_j|$ batches that can be loaded by jobs of family $j$.  The sets $\Gamma_j$ constitute a partition of $\Omega$, meaning that $\cup_{j=1}^m \Gamma_j=\Omega$ and $\Gamma_i \cap
\Gamma_j = \emptyset$ for all $\{i,j\}\subset \Omega$.  
For any $j\in \mathcal{F}$, we introduce binary variables $x_{kt}$ that are one if batch $k\in \Gamma_j$ starts processing at time period $t \in H_j$, and zero otherwise,
and we also use variables~$y_{ik}$ that equal one if job $i\in N_j$ is assigned to batch $k\in \Gamma_j$, and zero otherwise. Finally, $C_i$ is a non-negative continuous variable that represents the completion time of job $i$, for $i\in N$.
\begin{align}
	\label{obj:TIF2017}
	\min \ &\displaystyle \sum_{i\in N} w_iC_i\\
	s.t. \ &\displaystyle \sum_{k\in \Gamma_j}y_{ik}= 1,
	&&\forall i \in N_j,j\in \mathcal{F},\label{constraint_01_TIF}\\
	&\displaystyle \sum_{i\in N_j} v_i y_{ik} \leq V,
	&&k\in \Gamma_j,j\in \mathcal{F},\label{constraint_02_TIF}\\
	& \sum_{t \in H_j} x_{kt} =1,
	&& k\in \Gamma_j, j\in \mathcal{F},\label{constraint_03_TIF}\\
	&\sum_{j \in \cal{F}}\sum_{k \in \Gamma_j} \sum_{\tau=\max\{t-q_j+1,1\}}^{\min\{t, |H_j|\}}
	x_{k\tau} \leq 1,
	&&t\in H,\label{constraint_04_TIF}\\
	&C_i \geq \sum_{t \in H_j} (t+q_j-1)x_{kt}-M(1-y_{ik}),
	&&i \in N_j, k\in \Gamma_j, j\in \mathcal{F}.\label{constraint_06_TIF}
\end{align}

The objective is to minimize the total weighted completion time.
Constraints~(\ref{constraint_01_TIF}) ensure that each job is assigned to one batch.
Constraints~(\ref{constraint_02_TIF}) impose the machine capacity on each batch.
Constraints~(\ref{constraint_03_TIF}) schedule each batch exactly once.
Constraints~(\ref{constraint_04_TIF}) guarantee that the machine will handle at most one batch at each time period.
Constraints~(\ref{constraint_06_TIF}) compute the job completion times, with $M$ a suitably large number.


\section{Batch generation}
\label{Appendix:batch generation}
We present a method to generate all feasible batches per family (that is, those that respect the capacity constraints), which is loosely based on \cite{stork2001}'s method for generating minimal forbidden sets in resource-constrained project scheduling.  For a given family~$j$, we enumerate all feasible subsets of $N_j$ by means of a tree $T$, where each node $w$ of $T$ corresponds to a job $i\in N_j$, except for the root node. If  node $w$ corresponds to job $i$ then it will have a potential child node for each job $i'=i+1,\ldots, |N_j|$. The root node has $|N_j|$ potential child nodes, corresponding to $|N_j|$ jobs. Each node~$w$ defines a batch $s \in S_j$ that is obtained by traversing the tree from node $w$ to the root, such that the batch $s$ includes all jobs associated with the nodes along the path.

During the construction of the tree, we prune a node as soon as the total size of the jobs in the corresponding batch violates the machine capacity. The enumeration uses a depth-first strategy.  Upon creating a node, we first check the capacity constraint.  If it is violated then we prune, otherwise the batch is feasible and we branch into the child nodes. If a node is feasible and has no further descendants, it is a leaf node.

We refer to this procedure as BG, and provide a pseudocode description in Algorithm~\ref{algorithm:Stork}. Here, let $x_i=1$ if the batch~$s$ includes job~$i\in N_j$, 0 otherwise. The value $V_s$ keeps track of the total size of all jobs in batch~$s$. We initialize $x_k=0$ for all  $k\in N_j$, $i=1$ and $V_s=0$.
An example is included in Figure~\ref{fig:StorkExample} for job set $N_j=\{1,2,3,4\}$, assuming job size $v_k=k$ for $k\in N_j$ (i.e., job $1$ has $v_1=1$) and the machine capacity $V = 5$.
\begin{algorithm}
	\caption{BG}
	\begin{algorithmic}[1]
		\If {$i > |N_j|$}
		\State Create new pattern $s$
		\For {$k \in N_j$}
		\If{$x_k=1$}
		
		\State add job $k$ to $s$
		\EndIf
		\State\textbf{ end if}
		\EndFor
		\State\textbf{ end for}
		\If{$s$ is not empty}
		\State add $s$ to $S_j$
		\EndIf
		\State\textbf{ end if}
		\Else{
			\If{$V_s + v_i> V$}
			\State $x_i=0$
			\State \text{BG($i+1$, $V_s$)}
			\Else{
				\State $x_i=1$
				\State $V_s=V_s+ v_i$
				\State \text{BG($i+1$, $V_s$)}
				\State $x_i=0$
				\State $V_s=V_s - v_i$
				\State \text{BG($i+1$, $V_s$)}
				\EndIf
				\State\textbf{ end if}
			}
		}
		\EndIf
		\State\textbf{ \quad end if}
	\end{algorithmic}
	\label{algorithm:Stork}
\end{algorithm}

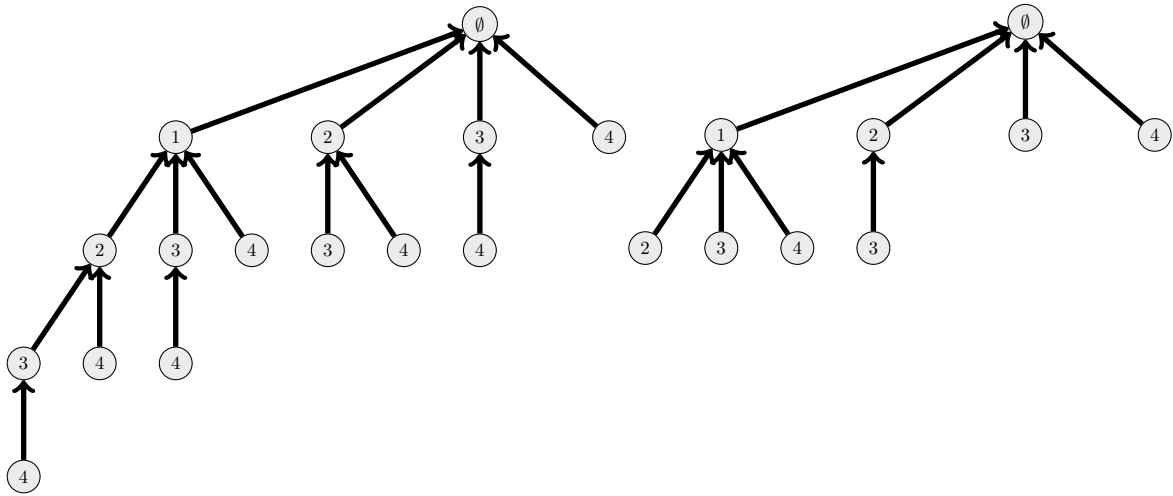
\begin{figure}
	 \begin{subfigure}{0.45\textwidth}
			\begin{tikzpicture}
				[every node/.style={scale=0.7,draw,shape=circle,fill=gray!15}]
				\node (n0) at (0, 0) {$\emptyset$};
				\node (n1) at (-4, -1.5) {1};
				\node (n2) at (-2, -1.5) {2};
				\node (n3) at (0, -1.5) {3};
				\node (n4) at (1.7, -1.5) {4};
				\draw[->,line width=2pt] (n1) to (n0);
				\draw[->,line width=2pt] (n2) to (n0);
				\draw[->,line width=2pt] (n3) to (n0);
				\draw[->,line width=2pt] (n4) to (n0);	
				\node (n5) at (-5,-3) {2};
				\node (n6) at (-4,-3) {3};
				\node (n7) at (-3,-3) {4};
				\draw[->,line width=2pt] (n5) to (n1);
				\draw[->,line width=2pt] (n6) to (n1);
				\draw[->,line width=2pt] (n7) to (n1);
				\node (n8) at (-6,-4.5) {3};
				\node (n9) at (-5,-4.5) {4};
				\draw[->,line width=2pt] (n8) to (n5);
				\draw[->,line width=2pt] (n9) to (n5);
				\node (n10) at (-4,-4.5) {4};
				\draw[->,line width=2pt] (n10) to (n6);
				\node (n14) at (-6,-6) {4};
				\draw[->,line width=2pt] (n14) to (n8);
				\node (n11) at (-2,-3) {3};
				\node (n12) at (-1,-3) {4};
				\draw[->,line width=2pt] (n11) to (n2);
				\draw[->,line width=2pt] (n12) to (n2);
				\node (n13) at (0,-3) {4};
				\draw[->,line width=2pt] (n13) to (n3);
			\end{tikzpicture}
	\caption{The complete tree for $N_j$ before pruning}
	\end{subfigure}
	\begin{subfigure}{0.45\textwidth}
			\begin{tikzpicture}
				[stay/.style={scale=0.7,draw,shape=circle,fill=gray!15},
				deleted/.style={scale=0.7}
				]
				\node[stay] (n0) at (0, 0) {$\emptyset$};
				\node[stay] (n1) at (-4, -1.5) {1};
				\node[stay] (n2) at (-2, -1.5) {2};
				\node[stay] (n3) at (0, -1.5) {3};
				\node[stay] (n4) at (1.7, -1.5) {4};
				\draw[->,line width=2pt] (n1) to (n0);
				\draw[->,line width=2pt] (n2) to (n0);
				\draw[->,line width=2pt] (n3) to (n0);
				\draw[->,line width=2pt] (n4) to (n0);
				\node[stay] (n5) at (-5,-3) {2};
				\node[stay] (n6) at (-4,-3) {3};
				\node[stay] (n7) at (-3,-3) {4};
				\draw[->,line width=2pt] (n5) to (n1);
				\draw[->,line width=2pt] (n6) to (n1);
				\draw[->,line width=2pt] (n7) to (n1);

				\node[stay] (n11) at (-2,-3) {3};
			
				\draw[->,line width=2pt] (n11) to (n2);

				\node[deleted] (n14) at (-6,-6.2) {};
			\end{tikzpicture}
	\caption{The tree of all feasible batches}
	\end{subfigure}
	\caption{Trees for the example}
	\label{fig:StorkExample}
\end{figure}




\section{Computational performance of TIF, SPF, and B\&P on K2008 and K2008u instances}
\label{Appendix:Experimental data}

The results are in the Tables~\ref{table:Non-unit}-\ref{table:Unit}.

\begin{landscape}
\begin{table}
	\begin{threeparttable}
	\centering
	\small
	\setlength\tabcolsep{3pt}
	\caption{Computational performance of TIF, SPF, and B\&P on K2008 instances with $n \in \{40, 60, 80\}$; `time' indicates the average CPU times in seconds, `opt' the number of solved instances (out of 10) within the 1200-seconds time limit, and `nodes' the average number of nodes computed by B\&P. In the `opt'-columns, the numbers in brackets indicate the number of instances that cannot be solved because of memory problems, or because the time required to build the model exceeds the time limit.}
	\begin{tabular}{cc*{21}{R}}
		\toprule[1pt]
		& & \multicolumn{7}{c}{$n = 40$} &  \multicolumn{7}{c}{$n = 60$} &  \multicolumn{7}{c}{$n = 80$} \\ \cmidrule(lr){3-9} \cmidrule(lr){10-16} \cmidrule(lr){17-23}
		& & \multicolumn{2}{c}{TIF} & \multicolumn{2}{c}{SPF} & \multicolumn{3}{c}{B\&P} & \multicolumn{2}{c}{TIF} & \multicolumn{2}{c}{SPF} & \multicolumn{3}{c}{B\&P} &  \multicolumn{2}{c}{TIF} & \multicolumn{2}{c}{SPF} & \multicolumn{3}{c}{B\&P} \\
		\cmidrule(lr){3-4}
		\cmidrule(lr){5-6}
		\cmidrule(lr){7-9}
		\cmidrule(lr){10-11}
		\cmidrule(lr){12-13}
		\cmidrule(lr){14-16}
		\cmidrule(lr){17-18}
		\cmidrule(lr){19-20}
		\cmidrule(lr){21-23}
		\multicolumn{1}{c}{$m$} & \multicolumn{1}{c}{$v_i$}
		& \multicolumn{1}{c}{time} & \multicolumn{1}{c}{opt} & \multicolumn{1}{c}{time} & \multicolumn{1}{c}{opt} & \multicolumn{1}{c}{time} & \multicolumn{1}{c}{opt} & \multicolumn{1}{c}{nodes}
		& \multicolumn{1}{c}{time} & \multicolumn{1}{c}{opt} & \multicolumn{1}{c}{time} & \multicolumn{1}{c}{opt} & \multicolumn{1}{c}{time} & \multicolumn{1}{c}{opt} & \multicolumn{1}{c}{nodes}
		& \multicolumn{1}{c}{time} & \multicolumn{1}{c}{opt} & \multicolumn{1}{c}{time} & \multicolumn{1}{c}{opt} & \multicolumn{1}{c}{time} & \multicolumn{1}{c}{opt} & \multicolumn{1}{c}{nodes}\\
		\midrule
		2 & [1, 10] & 373.05 & 7     & 50.11 & 10    & 31.87 & 9     & 107.67
		& 692.57 & 2     & 399.95 & 9     & 415.48 & 6     & 418.17 &
		\text{---} & 0     & \text{---} & 0(7)     & \text{---} & 0     & \text{---}\\
		4 & [1, 10] & 449.12 & 3     & 68.26 & 10    & 69.19 & 10    & 80.80 &
		\text{---} & 0     & 455.03 & 7     & 335.07 & 5     & 181.80 &
		\text{---} & 0     & 286.25 & 2     & 51.39 & 1     & 7.00 \\
		6 & [1, 10] & 1181.51 & 1     & 52.28 & 10    & 68.50 & 10    & 29.00 &
		\text{---} & 0     & 433.01 & 7     & 404.44 & 8     & 69.50 &
		\text{---} & 0     & 856.81 & 2     & 556.66 & 2     & 39.00\\
		10 & [1, 10] & \text{---} & 0     & 41.98 & 10    & 39.49 & 10    & 5.00 &
		\text{---} & 0     & 349.12 & 10    & 271.12 & 9     & 14.11 &
		\text{---} & 0     & 885.55 & 1     & 1024.56 & 1     & 19.00 \\
		\midrule
		2 & [2, 4] & 84.71 & 10    & 163.41 & 10    & 244.20 & 3     & 1095.67 &
		451.57 & 5     & \text{---} & 0(6)      & \text{---} & 0     & \text{---} &
		\text{---} & 0     & \text{---} & 0(10)     & \text{---} & 0     & \text{---}  \\
		4 & [2, 4] & 643.80 & 3     & 228.45 & 10    & 152.84 & 6     & 377.00 &
		\text{---} & 0     & 1117.07 & 1     & \text{---} & 0     & \text{---} &
		\text{---} & 0     & \text{---}& 0(5)     & \text{---} & 0     & \text{---} \\
		6 & [2, 4] & \text{---} & 0     & 121.26 & 10    & 262.69& 9     & 281.22 & \text{---}&
		0     & 447.32 & 3     & \text{---} & 0     &\text{---} &
		\text{---} & 0     & \text{---} & 0     & \text{---} & 0     & \text{---} \\	
		10 & [2, 4] & \text{---} & 0     & 36.25 & 10    & 144.47 & 9     & 60.33  &
		\text{---} & 0     & 664.06 & 9     & 615.69 & 1     & 85.00 &
		\text{---} & 0     & \text{---} & 0     & \text{---}& 0     & \text{---} \\\midrule
		2 & [4, 8] & 334.77 & 7     & 17.90 & 10    & 118.61 & 10    & 433.60  &
		1072.97 & 1     & 142.47 & 10    & 142.07 & 8     & 206.50 &
		\text{---} & 0     & 509.66 & 6     &51.80 & 1     & 13.00 \\		
		4 & [4, 8] & 364.50 & 2     & 76.69 & 10    & 40.18 & 10    & 43.00 &
		\text{---} & 0     & 341.54 & 9     & 390.61 & 9     & 152.33 &
		\text{---} & 0     & 475.61 & 8     & 544.49 & 6     & 111.67 \\
		6 & [4, 8] & 504.73 & 3     & 29.07 & 10    & 22.87 & 10    & 8.00 &
		\text{---} & 0     & 230.80 & 8     & 139.66 & 9     & 21.89 &
		\text{---} & 0     & 280.36 & 1     & 599.89 & 4     & 39.00 \\
		10 & [4, 8] & \text{---} & 0     & 34.36 & 10    & 46.19 & 10    & 4.00 &
		\text{---} & 0     & 289.80 & 7     & 324.47 & 10    & 11.80 &
		\text{---} & 0     & 608.52 & 4     & 681.93 & 6     & 11.00 \\
		\midrule
		\multicolumn{2}{l}{average/total} & 347.37 & 36    & 76.67 & 120   & 91.05 & 106   & 147.42 &
		589.49 & 8     & 376.45 & 80    & 301.74 & 65    & 115.75
		& \text{---} & 0     & 531.20 & 24    & 571.39 & 21    & 48.05 \\
		\bottomrule[1pt]
	\end{tabular}
	\label{table:Non-unit}
	\end{threeparttable}
\end{table}

\begin{table}[t]
	\centering
	\begin{threeparttable}
	\small
	\setlength\tabcolsep{3pt}
	\caption{Computational performance of TIF, SPF, and B\&P on K2008u instances; `time' indicates the average CPU times in seconds, `opt' the number of solved instances (out of 10) within the 1200-seconds time limit, and `nodes' the average number of nodes computed by B\&P. In the `opt'-columns, the numbers in brackets indicate the number of instances that cannot be solved because  of memory problems, or because the time required to build the model exceeds the time limit.}
	\begin{tabular}{ccRRRRRRRRRRRRRR}
		\toprule[1pt]
		& & \multicolumn{7}{c}{$n = 120$} &  \multicolumn{7}{c}{$n = 150$} \\ \cmidrule(lr){3-9} \cmidrule(lr){10-16}
		& & \multicolumn{2}{c}{TIF} & \multicolumn{2}{c}{SPF} & \multicolumn{3}{c}{B\&P}  & \multicolumn{2}{c}{TIF} & \multicolumn{2}{c}{SPF} & \multicolumn{3}{c}{B\&P} \\
		\cmidrule(lr){3-4}
		\cmidrule(lr){5-6}
		\cmidrule(lr){7-9}
		\cmidrule(lr){10-11}
		\cmidrule(lr){12-13}
		\cmidrule(lr){14-16}
		$m$ & $v_i$ & \multicolumn{1}{c}{time} & \multicolumn{1}{c}{opt} & \multicolumn{1}{c}{time} & \multicolumn{1}{c}{opt} & \multicolumn{1}{c}{time} & \multicolumn{1}{c}{opt} & \multicolumn{1}{c}{nodes} & \multicolumn{1}{c}{time} & \multicolumn{1}{c}{opt} & \multicolumn{1}{c}{time} & \multicolumn{1}{c}{opt} & \multicolumn{1}{c}{time} & \multicolumn{1}{c}{opt} & \multicolumn{1}{c}{nodes} \\
		\midrule
		2 & [1, 10] & 742.09 & 3     & 274.06 & 8(1)     & \text{---} & 0     & \text{---} & 65.80 & 1     & 379.24 & 1(5)     & \text{---} & 0     & \text{---} \\
		4 & [1, 10] & 48.85 & 1     & 74.43 & 10    & 78.82 & 2     & 264.00 & \text{---} & 0     & 137.24 & 10    & \text{---} & 0     & \text{---} \\
		6 & [1, 10] & \text{---} & 0     & 26.72 & 10    & 211.59 & 3     & 1651.00 & \text{---} & 0     & 76.06 & 10    & 947.03 & 1     & 3059.00 \\
		10 & [1, 10] & 931.69 & 1     & 4.07  & 10    & 160.59 & 6     & 1699.33 & 115.29 & 1     & 15.87 & 10    & 447.79 & 3     & 2540.33 \\ \midrule
		2 & [2, 4] & 519.24 & 5     & 685.39 & 1     & \text{---} & 0     &\text{---} & 55.93 & 1     & \text{---} & 0(10)     & \text{---}& 0     & \text{---} \\
		4 & [2, 4] & 252.00 & 1     & 377.74 & 9     & \text{---} & 0     & \text{---} & \text{---} & 0     & 845.90 & 1     & \text{---} & 0     & \text{---} \\
		6 & [2, 4] & \text{---} & 0     & 139.55 & 10    & \text{---}& 0     & \text{---} & \text{---}& 0     & 347.35 & 3     & \text{---} & 0     & \text{---} \\
		10 & [2, 4] & \text{---}& 0     & 22.61 & 10    & \text{---} & 0     & \text{---} & \text{---} & 0     & 74.46 & 10    & \text{---} & 0     & \text{---} \\ \midrule
		2 & [4, 8] & 187.43 & 6     & 13.83 & 10    & 210.53 & 2     & 1392.00 & \text{---} & 0     & 11.46 & 10    & \text{---} & 0     & \text{---}\\
		4 & [4, 8] & \text{---} & 0     & 4.92  & 10    & 195.54 & 3     & 2269.67 &\text{---} & 0     & 9.28  & 10    & 236.25 & 4     & 1929.00 \\
		6 & [4, 8] & 323.92 & 2     & 2.24  & 10    & 107.13 & 7     & 1439.57 & \text{---} & 0     & 8.07  & 10    & 100.41 & 3     & 1082.33 \\
		10 & [4, 8] & 1029.56 & 1     & 3.37  & 10    & 446.13 & 9     & 7014.56 & \text{---} & 0     & 5.94  & 10    & 290.14 & 6     & 3152.33 \\
		\midrule
		\multicolumn{2}{c}{average/total} & 442.85 & 20    & 85.14 & 108   & 235.27 & 32    & 3077.44  & 79.01 & 3     & 66.48 & 85    & 310.44 & 17    & 2385.71 \\
		\bottomrule[1pt]
	\end{tabular}
	\label{table:Unit}
	\end{threeparttable}
\end{table}

\end{landscape}
\end{appendices}

\end{document}